\documentclass[11pt,a4,reqno]{amsart}
\usepackage{amsfonts,amsmath,amssymb}
\usepackage{mathabx}
\usepackage{mathrsfs}
\usepackage{nicefrac}
\usepackage[latin1]{inputenc}
\usepackage[usenames,dvipsnames]{pstricks}

\usepackage[pagewise]{lineno}
\newtheorem{theorem}{Theorem}[section]

\newtheorem{lemma}[theorem]{Lemma}

\newtheorem{prop}[theorem]{Proposition}

\newcommand{\R}{\mathbb{R}}

\usepackage{graphicx}

\usepackage{times}

\usepackage[toc,page]{appendix}
\usepackage{enumitem} 
\begin{document}
\title[Stable recovery of Schr\"odinger equation]
{Stable recovery of a non-compactly supported coefficient of a Schr\"odinger equation on an infinite waveguide }


\author{Yosra Soussi}

\address{Universit\'e de Tunis El Manar, Ecole Nationale d'Ing\'enieurs de Tunis, ENIT-LAMSIN, B.P. 37, 1002 Tunis, Tunisia \newline Aix Marseille Universit\'e, Universit\'e de Toulon, CNRS, CPT, Marseille, France }

\email{yosra.soussi@enit.utm.tn}
%
%
%


\subjclass[2010]{Primary 35R30, Secondary: 35J15. } 
\keywords{Inverse problem, Stability estimate; Schr\"odinger equation; electric potential; Dirichlet-to-Neumann map; infinite cylindrical waveguide; partial data; Carleman's estimate}
\date{\today}

\maketitle

\begin{abstract}
We study the stability issue for the inverse problem of determining a coefficient appearing in a Schr\"odinger
equation defined on an infinite cylindrical waveguide. More precisely, we prove the stable recovery of some general class of non-compactly and non periodic coefficients appearing in an unbounded cylindrical domain. We consider both results of stability from full and partial boundary measurements associated with the so called Dirichlet-to-Neumann map.
\end{abstract}
\section{Introduction}
Let $ \Omega$ be an unbounded open connected set of $\R^3$ taking the form $\Omega := \omega \times \R$, with $\omega$ a $\mathcal{C}^3$ bounded open connected set of $\R^2$. Let $\partial \Omega =\partial \omega \times \R$ be the boundary of $\Omega$. We fix $q\in L^{\infty}(\Omega)$ such  that $0$ is not in the spectrum of the operator $- \Delta + q $ acting on $L^2(\Omega)$ with homogeneous Dirichlet boundary condition. Then, we consider the following boundary value problem:
\begin{equation}\label{1}
\left\lbrace
\begin{array}{l}
\text{$-\Delta u + q  u = 0 \quad $ in $ \Omega$ ,} \\
\text{$u = f \quad \, \quad \; \quad \quad \; $ on $ \partial \Omega $.}
\end{array}\right.
\end{equation}
Since $\Omega$ is unbounded, for $X=\omega$ or $X=\partial \omega$ and any $r>0$, we define the space $H^r (X \times \R)$ by $$H^r (X \times \R):= L^2 (X ; H^r(\R)) \cap L^2 (\R ; H^r(X))  .$$
We denote also by $H^{-r}(\partial \Omega)$ the dual space of $H^{r}(\partial \Omega)$. Combining \cite[Lemma 2.2]{14} with some classical lifting arguments (see e.g. \cite[Theorem 8.3, Chapter 1]{36}), which can be extended to infinite cylindrical domain $\Omega$ by using arguments similar to \cite[Lemma 2.2]{17}, we deduce that for $f \in H^{\frac{3}{2}}(\partial\Omega)$ problem $\eqref{1}$ admits a unique solution $u \in H^2 (\Omega)$. In the present paper, we consider the inverse problem of identifying the electric potential $q$ from measurements on the boundary of solutions of $\eqref{1}$. More precisely, our observations are given by the Dirichlet-to-Neumann (DN in short) map $\Lambda_q$ which is defined by
\begin{equation}\label{2}
\begin{array}{lll}
\Lambda_{q} : & H^{\frac{3}{2}} (\partial \Omega) &\longrightarrow \quad H^{\frac{1}{2}} (\partial \Omega)\\
&\quad f &\longmapsto \partial_\nu u :=  \nabla u \cdot \nu.
\end{array}
\end{equation} 
Here $\nu$ is the outward unit normal vector on $\partial \Omega$ which takes the form $$\nu (x' , x_3):=(\nu'(x'),0), \quad x=(x',x_3) \in \partial \Omega, \quad x':=(x_1,x_2)\in \omega,$$
with $\nu'$ is the outward unit normal vector on $\partial \omega$. Moreover, we use here the fact that the map $H^2(\Omega) \ni w \mapsto \partial_{\nu} w \in H^{\frac{1}{2}} (\partial \Omega)$ is a bounded operator which can be deduced from a combination of arguments of \cite[Theorem 8.3, Chapter 1]{36} and \cite[Lemma 2.2]{17}.
\\
Let $\Gamma_0 \subset \partial \omega $ be an arbitrary open set. The restriction $\Lambda'_{q}$ of $\Lambda_{q}$ on $\Gamma_0 \times \R$ is defined by
\begin{equation}\label{2'}
\begin{array}{lll}
\Lambda'_{q} : & H^{\frac{3}{2}} (\partial \Omega) &\longrightarrow \quad H^{\frac{1}{2}} (\Gamma_0 \times \R)\\
&\quad f &\longmapsto \partial_\nu u :=  \nabla u \cdot \nu_{ \vert{\Gamma_0 \times \R}}.
\end{array}
\end{equation}
Let us mention that the DN map $\Lambda_q$ can also be defined as a bounded operator from $H^{\frac{1}{2}}(\partial \Omega)$ to $H^{-\frac{1}{2}}(\partial \Omega)$. However, since the boundary of $\Omega$ is not bounded, we consider this definition of the DN map in order to simplify some argumentation in the paper. 
In this paper, we study the stable recovery of $q$ from knowledge of the full DN map $\Lambda_{q}$ and the partial DN map $\Lambda'_{q}$.
\subsection{Physical motivation}
The problem under consideration in this paper is related to the Calder\`on's problem \cite{8}. The latter can be seen as the determination of an electrical conductivity of a medium by making voltages and current measurements at the boundary of the medium. This problem is called the Electrical Impedance Tomography (EIT) problem (see \cite{37}) which has many applications including geophysical prospection \cite{38} and medical imaging by improving the early detection of breast cancer \cite{39}. \\
The statement of this problem on the infinite cylindrical domain (also called infinite waveguide) can be associated with problems of transmission to long distance, where the goal of our inverse problem is to determine an impurity that perturbs the guided propagation.
\subsection{Known results}
In 1980, Calder\'on considered in \cite{8} the first formulation of a problem which is now called the Calder\'on problem (see \cite{13,37} for an overview). In \cite{40}, the authors provided the first positive answer to this problem stated in terms of a uniqueness result from full boundary measurements. Since then, this problem has received much attention and this result has been improved in several ways. This includes some results with partial data like \cite{2,7,20,23,24} and some stability results stated in \cite{1,5,9,10,12,15}.\\
All the above mentioned results have been stated in a bounded domain. Several authors considered also the
recovery of coefficients for elliptic equations on an unbounded domain corresponding to a slab or an infinite cylindrical waveguide. For instance, we can mention the work of \cite{31,33,34,35} dealing with the unique recovery of several classes of compactly supported coefficients in a slab from boundary measurements. The stability issue for these classes of inverse problems has been addressed by \cite{11}. For results devoted to recovery of coefficients in an infinite cylindrical domain, we can first mention the work of \cite{16,17} where the stable recovery of coefficients periodic along the axis of an infinite cylindrical domain has been addressed. More recently, \cite{25,26} obtained the first results of recovery of non-compactly supported and non-periodic coefficients appearing in an elliptic equation on infinite cylindrical domain. The results of \cite{25,26} seems to be the first results of unique recovery of such a general class of electromagnetic potentials in an unbounded domain. In addition, the analysis of \cite{20, 21} extends also the
results of \cite{31,35} to the recovery of non-compactly supported coefficients in a slab. Finally, we refer to the works \cite{30, 4, 3, 6, 27, 29, 15, 18, 22} for the analysis of inverse problems of identifying non compactly supported coefficients and some related problems in an unbounded domain.
\subsection{Statement of the main results}
In this paper we consider the inverse problem of identifying the electric potential $q$ from boundary measurements. For this purpose, we need to define a set of admissible coefficients. In order to introduce the set of admissible unknown potentials, we fix $M \in (0 , + \infty )$ and $q_0 \in H^1 (\Omega ) \cap L^{\infty} ( \Omega)$ such that $\Vert q_0 \Vert_{L^{\infty} ( \Omega)} \leqslant M$. Then, we consider the set of admissible unknown potentials $\mathcal{Q}(M, q_0 )$ corresponding to the set of all $q \in H^1 (\Omega ) \cap L^{\infty} ( \Omega)$ satisfying the following conditions:
\begin{itemize}
\item[(i)] \begin{equation}\label{condnorm}
\|q - q_0 \|_{H^{1}(\Omega)} + \|q \|_{L^{\infty}(\Omega)} +  \displaystyle\int_\Omega (1 + \vert x_3 \vert) \vert (q - q_0) (x',x_3) \vert \, dx' dx_3\leqslant M, 
\end{equation}
\item[(ii)] \begin{equation}\label{cond2}
q(x) = q_0 (x), \quad x \in \partial \Omega,
\end{equation}
\item[(iii)] $0$ is not in the spectrum of the operator $-\Delta + q$ acting on $L^2(\Omega)$ with homogeneous Dirichlet boundary condition.
\end{itemize}
We start by proving the stable recovery of $q$ from measurements on the whole boundary. Our first main result can be stated as follows.  
\begin{theorem}\label{thm2}
Let $M>0$ and, for $j=1,2$, let $q_j \in \mathcal{Q}(M, q_0 )$. Then, there exists a positive constant $C$ depending only on $M$ and $\Omega$ such that 
\begin{equation}\label{17}
\big \Vert q_1 - q_2 \big \Vert_{L^2(\Omega)} \leqslant C\log\Big( 3+ \big\Vert \Lambda_{q_1}-\Lambda_{q_2} \big\Vert_{\mathcal{B}(H^{\frac{3}{2}} (\partial \Omega) , H^{\frac{1}{2}} (\partial \Omega))}^{-1}\Big)^{-\frac{1}{36}}.
\end{equation} 
Here and from now, for any Banach spaces $X,Y$, we denote by $\mathcal{B}(X,Y)$ the space of bounded operator from $X$ to $Y$ with the associated norm.
\end{theorem}
For $\big\Vert \Lambda_{q_1}-\Lambda_{q_2} \big\Vert_{\mathcal{B}(H^{\frac{3}{2}} (\partial \Omega) , H^{\frac{1}{2}} (\partial \Omega))} = 0$, the quantity on the right hand side of $\eqref{17}$ corresponds to 0. \\
Next, we give our partial data result with restriction of the measurement to some portion of the boundary. In
contrast to the previous full data result this partial data result requires some additional definitions and assumptions that we need to recall first. Let $\mathcal{W}_0 \subset \omega $ be an arbitrary neighborhood of the boundary $\partial \omega$ such that $\partial \mathcal{W}_0 = \partial \omega \cup \Gamma^{\sharp} $ with $\partial \omega \cap \Gamma^{\sharp} = \varnothing $. We assume that $\Gamma^\sharp$ is $\mathcal{C}^2$. Let $\Gamma_0 \subset \partial \omega \subset \partial \mathcal{W}_0$ be an arbitrary (not empty) open set of $\partial \omega$ and let $\mathcal{O}_0 = \mathcal{W}_0 \times \R$.
For a given $M>0$, we introduce the admissible set of coefficients
$$\mathcal{Q}'(M, q_0 ,\mathcal{O}_0) = \lbrace q \in \mathcal{Q}(M, q_0 ) \cap \mathcal{C}^2(\overline{\Omega}, \R^3); \, \Vert q \Vert_{\mathcal{C}^2(\overline{\Omega})} \leqslant M \text{ and } q = q_0 \text{ in } \mathcal{O}_0  \rbrace .$$
Then, our second main result can be stated as follows.
\begin{theorem}\label{thm3}
Let $M>0$ and, for $j=1,2$, let $q_j \in \mathcal{Q}'(M, q_0 ,\mathcal{O}_0)$. Then, there exists a positive constant $C$ depending only on $M$ and $\Omega$ such that 
\begin{equation}\label{555}
\big \Vert q_1 - q_2 \big \Vert_{L^2(\Omega)} \leqslant C \log \Big( 3 + \big\Vert \Lambda'_{q_1}-\Lambda'_{q_2} \big\Vert_{\mathcal{B}(H^{\frac{3}{2}} (\partial \Omega) , H^{\frac{1}{2}} (\Gamma_0 \times \R))}^{-1}\Big)^{-\frac{1}{36}}.
\end{equation} 
\end{theorem}
\subsection{Comments about our results}
To the best of our knowledge, Theorem \ref{thm2} is the first result of stable recovery of coefficients appearing in an elliptic equation which are neither compactly supported nor periodic. Only some uniqueness results have been obtained so far for the recovery of this class of coefficients (see \cite{25}). Indeed, one can only find in the mathematical literature works like \cite{11,16,17} devoted to the stable recovery of compactly supported or periodic coefficients in an unbounded domain. \\ \indent
In addition to the result of Theorem \ref{thm2}, we obtain a partial data result in the spirit
of \cite{5} for unbounded cylindrical domains stated in Theorem \ref{thm3}. Namely, by assuming that the coefficient under consideration is known close to the boundary, we show that the result of Theorem \ref{thm2} remain valid with measurements restricted to any portion of $\partial \Omega$ of the form $\Gamma \times \R$  with $\Gamma$ an arbitrary open subset of $\partial \omega$. We prove this last result by combining the tools of Theorem \ref{thm2} with a Carleman estimate for elliptic equations on unbounded cylindrical domain. Since, in contrast to all other similar works, the Carleman estimate required for Theorem \ref{thm3} is stated in an unbounded domain, we prove its derivation in the appendix by using a separation of variables argument. The stability estimate of Theorem \ref{thm3} can be compared to the one obtained by \cite{3}, for dynamical Schr\"odinger equations, in terms of restriction of the measurements. \\ \indent
Let us observe that, in contrast to the uniqueness results of \cite{25}, the stability results of Theorem  \ref{thm2},  \ref{thm3} require more careful estimates and some extra restriction of the class of coefficients under consideration. Moreover, in contrast to stability results in bounded domains (see e.g. \cite{1,2,11,12}), we are not able to obtain our stability estimate from the assumption of \cite{25} where the uniqueness for our inverse problem has been proved. Indeed, in order to derive the stability estimates \eqref{17} and \eqref{555}, we need to consider the extra assumption $\eqref{condnorm}$ which describes the behavior at infinity of the admissible coefficients.
\subsection{Outline}
This paper is organized as follows. As a first step, we introduce the complex geometrical optics (CGO in short) solutions of our problem in Section 2. Relying on the properties of these particular solutions and on the weak unique continuation property stated in the Appendix, we can prove that the electric potential depends stably on the global Dirichlet-to-Neumann map in the third section and on the partial one in the fourth section.
\setcounter{equation}{0}
\section{Construction of particular solutions}
In order to solve our inverse problem, we first borrow from \cite{25} the construction of particular solutions called CGO solutions. Let us start by fixing some notations. Let $\theta \in \mathbb{S}^1$, $\xi=(\xi',\xi_3)\in \R^3$ such that $\xi' \in \theta^\perp\setminus{\lbrace 0 \rbrace}$, $\xi_3 \in \R \setminus{\lbrace 0 \rbrace} $ and $ \eta \in \mathbb{S}^2$ given by $\eta=\frac{\Big(\xi',-\dfrac{\vert\xi'\vert^2}{\xi_3}\Big)}{\sqrt{\vert\xi'\vert^2 + \dfrac{\vert\xi'\vert^4}{\xi_3^2}} }$. In particular, we have $$\eta \cdot \xi = (\theta , 0) \cdot \xi = (\theta , 0) \cdot \eta = 0.$$
We consider the function $\chi \in \mathcal{C}_0^\infty ( (-2 , 2 ) , [0,1])$ such that $\chi = 1$ on $[-1,1]$. Since we are dealing with an unbounded domain, this cut-off function is introduced by Kian \cite{25} in order to insure the square integrability property of the CGO solutions. He also assumed that the principal part of the CGO solutions propagates along the axis of the waveguide and he used several arguments such as separation of variables and suitable Fourier decomposition of operators. By extending the work of \cite{25} we obtain the following result.
\begin{prop}\label{thm1}
Let $R>1$, let $q\in L^\infty (\Omega)$ be such that $\eqref{condnorm}$ is fulfilled with $q=q_j$ and let $\xi =(\xi',\xi_3) \in \R^2 \times \R$ be such that 
\begin{equation}\label{cc}
\vert \xi \vert \leqslant R , \quad \vert \xi' \vert > 0 , \quad \vert \xi_3 \vert > 0.
\end{equation}
 There exists $\rho_0 >1$, depending on $M$ and $\Omega$, such that for any $\rho >\rho_0 R^{16}$, the equation $-\Delta u + q  u = 0$ has a solution $u \in H^2(\Omega)$ given by
$$ u(x',x_3)=e^{-\rho\theta\cdot x'}\Big(e^{i\rho\eta\cdot x}\chi\big( \rho^{-\frac{1}{4}}x_3 \big)  e^{-i\xi\cdot x} + r_\rho (x) \Big),$$
with $r_\rho \in H^2(\Omega)$ satisfying
\begin{equation}\label{3}
\rho^{-1} \Vert r_\rho \Vert_{H^2(\Omega)} + \rho \Vert r_\rho \Vert_{L^2(\Omega)} \leqslant C R^2 \rho^{\frac{7}{8}},
\end{equation}
where $C>0$ depends only on $\Omega$ and $M$.
\end{prop}
In contrast to the construction of \cite{25}, in Proposition \ref{thm1} we give also some precise estimate of the CGO by taking into account the dependency with respect to the frequency $\xi$. These estimates will play an important role in the proof of Theorem \ref{thm2}. By a simple computation, it is easy to check that $u$ is a solution of $-\Delta u + q  u = 0$ if and only if $r_\rho$ solves
\begin{equation}\label{4}
P_{-\rho} r_\rho = -qr_\rho-e^{\rho\theta\cdot x'}(-\Delta + q)e^{-\rho\theta\cdot x'}e^{i\rho\eta\cdot x}\chi\big( \rho^{-\frac{1}{4}}x_3 \big) e^{-i\xi\cdot x},
\end{equation}
with $P_s:= -\Delta - 2s \theta \cdot \nabla' - s^2$, $s\in \R$.
Let us consider the equation 
\begin{equation}\label{5}
P_{-\rho} y(x)=F(x); \quad x \in \Omega.
\end{equation}
Then we have the following lemma proved in \cite{25}.
\begin{lemma}[Lemma 2.4, \cite{25}]\label{lem2}
For every $\rho>1$ there exists a bounded operator $E_{\rho} : L^{2}(\Omega) \rightarrow L^{2}(\Omega)$ such that 
\begin{equation}\label{11}
P_{-\rho} E_{\rho} F=F, \quad F \in L^{2}(\Omega),
\end{equation}
\begin{equation}\label{12}
\left\|E_{\rho}\right\|_{\mathcal{B}\left(L^{2}(\Omega)\right)} \leqslant C \rho^{-1},
\end{equation}
\begin{equation}\label{13}
E_{ \rho} \in \mathcal{B}\left(L^{2}(\Omega) ; H^{2}(\Omega)\right)
\end{equation}
and \begin{equation}\label{14}
\left\|E_{ \rho}\right\|_{\mathcal{B}\left(L^{2}(\Omega) ; H^{2}(\Omega)\right)}\leqslant C \rho,
\end{equation}
with $C>0$ depending only on $\Omega$.
\end{lemma}
Armed with this lemma we are know in position to complete the proof of Proposition \ref{thm1}.
\begin{proof}[Proof of Proposition $\ref{thm1}$]
First, we notice that \begin{multline}\label{15}
-e^{\rho \theta \cdot x^{\prime}}(-\Delta+q) e^{-\rho \theta \cdot x^{\prime}} e^{i \rho \eta \cdot x} \chi\left(\rho^{-\frac{1}{4}} x_{3}\right) e^{-i \xi \cdot x} \\=-\Big(\left(|\xi|^{2}+q\right) \chi\left(\rho^{-\frac{1}{4}} x_{3}\right)-2 i \eta_{3} \rho^{\frac{3}{4}} \chi^{\prime}\left(\rho^{-\frac{1}{4}} x_{3}\right)+2 i \xi_{3} \rho^{-\frac{1}{4}} \chi^{\prime}\left(\rho^{-\frac{1}{4} }x_{3}\right)\\-\rho^{-\frac{1}{2}} \chi^{\prime \prime}\left(\rho^{-\frac{1}{4}} x_{3}\right)  e^{i \rho \eta \cdot x} \Big) e^{-i \xi \cdot x}.
\end{multline}
On the other hand,
one can check that 
\begin{equation}\label{estchi}
\left\|\chi\left(\rho^{-\frac{1}{4} }x_{3}\right)\right\|_{L^{2}(\Omega)}+\left\|\chi^{\prime}\left(\rho^{-\frac{1}{4}} x_{3}\right)\right\|_{L^{2}(\Omega)}+\left\|\chi^{\prime \prime}\left(\rho^{-\frac{1}{4}} x_{3}\right)\right\|_{L^{2}(\Omega)} \leqslant C \rho^{\frac{1}{8}},
\end{equation}
with $C$ depending only on $\Omega$.
By combining the above arguments with $(\ref{15})$, we get
\begin{align*}
&\left\|-e^{\rho \theta \cdot x^{\prime}}(-\Delta+q) e^{-\rho \theta \cdot x^{\prime}} e^{i \rho \eta \cdot x} \chi\left(\rho^{-\frac{1}{4}} x_{3}\right) e^{-i \xi \cdot x}\right\|_{L^{2}(\Omega)} \\ & \leqslant C\left(\left(|\xi|^{2}+\|q\|_{L^{\infty}(\Omega)}\right) \rho^{\frac{1}{8}}+2\left|\eta_{3}\right| \rho^{\frac{7}{8}}+2\left|\xi_{3}\right| \rho^{-\frac{1}{8}}+\rho^{-\frac{3}{8}}\right) .
\end{align*}
Combining this with $\eqref{condnorm}$ and $\eqref{cc}$, we obtain
\begin{equation}\label{16}
\left\|-e^{\rho \theta \cdot x^{\prime}}(-\Delta+q) e^{-\rho \theta \cdot x^{\prime}} e^{i \rho \eta \cdot x} \chi\left(\rho^{-\frac{1}{4}} x_{3}\right) e^{-i \xi \cdot x}\right\|_{L^{2}(\Omega)} \leqslant C R^2 \rho^{\frac{7}{8}},
\end{equation}
with $C>0$ depending on $\Omega$ and $M$. By Lemma $\ref{lem2}$, $(\ref{4})$ could be written as  
\begin{equation}\label{eqq}
r_{\rho}= T_{\rho} r_{\rho} + G_{\rho} ,
\end{equation}  
with $ T_{\rho} f = - E_{\rho} (qf)$ ; $f \in L^2(\Omega)$ and $G_{\rho}= -E_{\rho}\left(e^{\rho \theta \cdot x^{\prime}}(-\Delta+q) e^{-\rho \theta \cdot x^{\prime}} e^{i \rho \eta \cdot x} \chi\left(\rho^{-\frac{1}{4}} x_{3}\right) e^{-i \xi \cdot x}\right) $.\\
In view of $(\ref{12})$ and $(\ref{16})$, we get 
$$ 
\|T_{\rho} h\|_{L^{2}(\Omega)} \leqslant C \rho^{-1}\|h\|_{L^{2}(\Omega)} \quad \text{and} \quad
\|G_{\rho} \|_{L^{2}(\Omega)} \leqslant C R^2 \rho^{-\frac{1}{8}}, $$
with $C>0$ depending on $\Omega$ and $M$. Therefore, fixing $\rho_0 = 2 C$ for $\rho  > \rho_0$, we can define $$r_{\rho} = (Id - T_{\rho})^{-1} G_{\rho} = \sum_{k=0}^{+\infty} T_{\rho}^k G_{\rho} $$ which satisfies \eqref{eqq}. Moreover, from $(\ref{12})$, $(\ref{13})$ and $(\ref{14})$, we deduce that, for $\rho > \rho_0 R^{16}$, $r_\rho \in H^2(\Omega)$ satisfies the decay property $(\ref{3})$. 
\end{proof}
Henceforth, we can consider the solutions $u_j \in H^2(\Omega)$ of $-\Delta u_j + q_j  u_j  = 0$ taking the forms
\begin{equation}\label{18}
u_1(x',x_3)=e^{-\rho\theta\cdot x'}\Big(e^{i\rho\eta\cdot x}\chi\big( \rho^{-\frac{1}{4}}x_3 \big)  e^{-i\xi\cdot x} + r_1 (x) \Big)
\end{equation}
and 
\begin{equation}\label{20}
u_2(x',x_3)=e^{\rho\theta\cdot x'}\Big(e^{-i\rho\eta\cdot x}\chi\big( \rho^{-\frac{1}{4}}x_3 \big) + r_2 (x) \Big),
\end{equation}
with $r_1 \in H^2(\Omega)$ and  $r_2 \in H^2(\Omega)$ satisfying
\begin{equation}\label{19}
\rho^{-1} \Vert r_1 \Vert_{H^2(\Omega)} + \rho \Vert r_1 \Vert_{L^2(\Omega)} \leqslant C R^2 \rho^{\frac{7}{8}}
\end{equation}
and 
\begin{equation}\label{21}
\rho^{-1} \Vert r_2 \Vert_{H^2(\Omega)} + \rho \Vert r_2 \Vert_{L^2(\Omega)} \leqslant C \rho^{\frac{7}{8}}.
\end{equation}
\setcounter{equation}{0}
\section{Stability on the whole boundary}
This section is devoted to the proof of Theorem \ref{thm2}. For this purpose, we fix $q_j \in \mathcal{Q} (M, q_0)$, $j = 1, 2$. We recall that since $q:=q_2 - q_1 = 0 $ on $ \partial \Omega $, we can extend $q$ to a $H^1(\R^3)$ vector by assigning it the value $0$ outside of $\Omega$ and we will refer to the extension by $q$. We recall also that the Proposition $\ref{thm1}$ guarantees the existence of solutions $u_j \in H^2(\Omega)$ to the equation $(-\Delta + q_j)u_j = 0$ in $\Omega$ given by $(\ref{18})$ and $(\ref{20})$. These solutions satisfy the following property. 
\begin{lemma}\label{lem3}
There exists $C>0$ such that for $R>1$, $\rho> \rho_0 R^{16}$ and $\xi \in \R^3$ satisfying $\eqref{cc}$, the following estimates  
 $$\Vert u_1 \Vert_{H^2(\Omega)} \leqslant C  e^{(D+1)\rho} \qquad \text{ and } \qquad \Vert u_2 \Vert_{H^2(\Omega)} \leqslant C  e^{(D+1)\rho}$$
hold true for any solutions $u_1$ and $u_2$ given by $(\ref{18})$ and $(\ref{20})$.\\
Here $D:= \underset{x' \in \overline{\omega}}{sup} \vert x'\vert$ and $\rho_0$ is the constant appearing in Proposition \ref{thm1}.
\end{lemma}
\begin{proof}
Using the expression of $u_1$, we can easily deduce that 
$$\Vert u_1 \Vert_{L^2(\Omega)} \leqslant \Vert e^{-\rho\theta\cdot x'} \Vert_{L^\infty(\Omega)} \Vert e^{i\rho\eta\cdot x}\chi\big( \rho^{-\frac{1}{4}}x_3 \big)  e^{-i\xi\cdot x} + r_1 (x) \Vert_{L^2(\Omega)}.$$
Setting $D:= \underset{x' \in \overline{\omega}}{sup} \vert x'\vert$, we get
$$\Vert u_1 \Vert_{L^2(\Omega)} \leqslant e^{D\rho} \Big( \Vert \chi\big( \rho^{-\frac{1}{4}}x_3 \big) \Vert_{L^2(\Omega)}   + \Vert r_1 (x) \Vert_{L^2(\Omega)} \Big). $$
$(\ref{19})$ and $(\ref{estchi})$ with the fact that $1<R<\rho$ lead to the following estimate
$$\Vert u_1 \Vert_{L^2(\Omega)} \leqslant C e^{D\rho} \Big( \rho^{\frac{1}{8}} +R^2\rho^{-\frac{1}{8}} \Big)  \leqslant C \rho^3 e^{D\rho} \leqslant C e^{(D+1)\rho}. $$
By simple computations of $\nabla u_1$ and $\partial_{x_i}\partial_{x_j}$, for $i,j=1,2,3$ and by the same arguments used previously, we obtain
\begin{align*}
\Vert \nabla u_1 \Vert_{L^2(\Omega)} & \leqslant  e^{D\rho} \Big( (\rho + \vert \xi \vert) \Vert \chi\big( \rho^{-\frac{1}{4}}x_3 \big) \Vert_{L^2(\Omega)} + \rho  \Vert \chi'\big( \rho^{-\frac{1}{4}}x_3 \big) \Vert_{L^2(\Omega)} + \rho \Vert  r_1 \Vert_{L^2(\Omega)} + \Vert \nabla r_1 \Vert_{L^2(\Omega)}\Big) \\
& \leqslant C e^{(D+1)\rho}
\end{align*}
and 
\begin{align*}
\Vert \partial_{x_i}\partial_{x_j} u_1 \Vert_{L^2(\Omega)} & \leqslant e^{D\rho} \Big( (\rho^2 + \vert \xi \vert^2) \Vert \chi\big( \rho^{-\frac{1}{4}}x_3 \big) \Vert_{L^2(\Omega)} + (\rho + \vert \xi \vert) \Vert \chi'\big( \rho^{-\frac{1}{4}}x_3 \big) \Vert_{L^2(\Omega)} \\
& + \rho^{-\frac{1}{2}} \Vert \chi''\big( \rho^{-\frac{1}{4}}x_3 \big) \Vert_{L^2(\Omega)} + \rho^2 \Vert  r_1 \Vert_{L^2(\Omega)} + \rho \Vert \nabla r_1 \Vert_{L^2(\Omega)} +  \Vert \partial_{x_i}\partial_{x_j} r_1 \Vert_{L^2(\Omega)} \Big) \\
& \leqslant C e^{(D+1)\rho}.
\end{align*}
In the same way, we get
$$\Vert u_2 \Vert_{L^2(\Omega)} \leqslant C e^{(D+1)\rho} \qquad \text{ and } \qquad \Vert  u_2 \Vert_{H^2(\Omega)} \leqslant C e^{(D+1)\rho} . $$
This completes the proof.
\end{proof}
Now we are able to prove the first main result of this paper.
\begin{proof}[Proof of Theorem $\ref{thm2}$]
In all this proof $C > 0$ is a constant depending only on $\Omega$ and $M$ that may change from line to line. Fix $R > 1$, $\rho> \rho_0 R^{16}$ and $\xi \in \R^3$ satisfying $\eqref{cc}$, with $\rho_0 > 1$ the constant appearing in Proposition $\ref{thm1}$. For $j=1,2$, consider also $u_j \in H^2 (\Omega )$ a solution of $- \Delta u_j + q_j u_j = 0 $ on $\Omega $ taking the form $\eqref{18}-\eqref{20}$ with $r_j$ satisfying $\eqref{19}-\eqref{21}$. Consider the following boundary value problem
\begin{equation}\label{22}
\left\lbrace
\begin{array}{l}
\text{$-\Delta w + q_1  w= 0 \quad $ in $ \Omega$ ,} \\
\text{$w = u_2:=h \quad \; \;  \quad \; $ on $ \partial \Omega $.}
\end{array}\right.
\end{equation}
Since $q_1 \in \mathcal{Q}(M, q_0)$, we know that $0$ is not in the spectrum of $-\Delta + q_1$ acting on $L^2(\Omega)$ with Dirichlet boundary condition. Therefore, \eqref{22} admits a unique solution $w \in H^2( \Omega)$. Fixing $u = w - u_2 \in H^2( \Omega)$, we deduce that $u$ solves
\begin{equation}\label{23}
\left\lbrace
\begin{array}{l}
\text{$-\Delta u + q_1  u = q u_2 \quad $ in $ \Omega$ ,} \\
\text{$u = 0 \quad  \quad \quad \quad \quad \quad \; $ on $ \partial \Omega $.}
\end{array}\right.
\end{equation}
Applying Green's Formula, we get 
\begin{align}\label{24}
\int_\Omega q u_2 u_1 \, dx & = \int_\Omega - \Delta u u_1 \, dx + \int_\Omega q_1 u u_1 \, dx \nonumber \\
& = - \int_{\partial \Omega} \partial_\nu u u_1 \, d\sigma_x + \int_{\partial \Omega} \partial_\nu u_1 u \, d\sigma_x \nonumber \\
& = - \int_{\partial \Omega} \big( \Lambda_{q_1}-\Lambda_{q_2} \big) h u_1 \, d\sigma_x .
\end{align}
Note that here and from now on, since $\Omega$ and $\partial \Omega$ are unbounded, we use \cite[Lemma 3.1]{25} in order to extend the usual Green's Formula to our framework. Moreover, we have 
\begin{equation}\label{25}
q u_1 u_2 = q e^{-i\xi\cdot x} \chi^2\big( \rho^{-\frac{1}{4}}x_3 \big) + z_\rho, 
\end{equation}
with 
$$z_\rho = q \Big[ e^{i\rho\eta\cdot x}e^{-i\xi\cdot x}\chi\big( \rho^{-\frac{1}{4}}x_3 \big) r_2 +  e^{-i\rho\eta\cdot x} \chi\big( \rho^{-\frac{1}{4}}x_3 \big) r_1 + r_1 r_2 \Big]. $$
Then, we have
\begin{align*}
\int_\Omega \vert z_\rho \vert & \leqslant \int_\Omega \vert q \vert \vert\chi\big( \rho^{-\frac{1}{4}}x_3 \big)\vert \vert r_1\vert \, dx + \int_\Omega \vert q \vert \vert\chi\big( \rho^{-\frac{1}{4}}x_3 \big)\vert \vert r_2\vert \, dx + \int_\Omega \vert q \vert  \vert r_1 r_2\vert \, dx ,\\
& := I_1 + I_2 + I_3.
\end{align*}
Using $(\ref{3})$ and the fact that $\| q \|_{L^{\infty}(\Omega)}\leqslant 2M$, we get
$$I_1 \leqslant \Vert q \Vert_{L^{2}(\Omega)} \Vert \chi\big( \rho^{-\frac{1}{4}}x_3 \big) \Vert_{L^{\infty}(\Omega)} \Vert r_1 \Vert_{L^{2}(\Omega)} \leqslant C R^2\rho^{-\frac{1}{8}} ,$$
$$I_2 \leqslant \Vert q \Vert_{L^{2}(\Omega)} \Vert \chi\big( \rho^{-\frac{1}{4}}x_3 \big) \Vert_{L^{\infty}(\Omega)} \Vert r_2 \Vert_{L^{2}(\Omega)} \leqslant C \rho^{-\frac{1}{8}} $$
and 
$$I_3 \leqslant \Vert q \Vert_{L^{\infty}(\Omega)} \Vert r_1 \Vert_{L^{2}(\Omega)}\Vert r_2 \Vert_{L^{2}(\Omega)} \leqslant CR^2 \rho^{-\frac{1}{4}} .$$
Thus, 
\begin{equation}\label{26}
\int_\Omega \vert z_\rho \vert \leqslant C R^2 \rho^{-\frac{1}{8}}.
\end{equation}
On the other hand, $(\ref{24})$ and $(\ref{25})$ give
$$\int_\Omega q  \chi^2\big( \rho^{-\frac{1}{4}}x_3 \big)e^{-i\xi\cdot x} \, dx  = - \int_{\partial \Omega} \big( \Lambda_{q_1}-\Lambda_{q_2} \big) h u_1 \, d\sigma_x -\int_\Omega z_\rho \, dx. $$
Thus, by $(\ref{26})$, we get 
\begin{align*}
\Big\vert \int_\Omega q  \chi^2\big( \rho^{-\frac{1}{4}}x_3 \big)e^{-i\xi\cdot x} \, dx \Big\vert & \leqslant CR^2 \rho^{-\frac{1}{8}} + C \big\Vert \Lambda_{q_1}-\Lambda_{q_2} \big\Vert \Vert u_2 \Vert_{H^{\frac{3}{2}}(\partial \Omega)} \Vert u_1\Vert_{L^2(\partial \Omega)} \\
& \leqslant C\Big( R^2 \rho^{-\frac{1}{8}} + \big\Vert \Lambda_{q_1}-\Lambda_{q_2} \big\Vert \Vert u_2 \Vert_{H^{2}(\Omega)} \Vert u_1\Vert_{H^{2}(\Omega)} \Big) \\
& \leqslant C R^2 \Big( \rho^{-\frac{1}{8}} + \big\Vert \Lambda_{q_1}-\Lambda_{q_2} \big\Vert e^{2(D+1)\rho}  \Big).
\end{align*}
Then 
\begin{equation}\label{27}
\Big\vert \int_{\R^3} q(x)  \chi^2\big( \rho^{-\frac{1}{4}}x_3 \big)e^{-i\xi\cdot x} \, dx \Big\vert \leqslant C  \Big( R^2 \rho^{-\frac{1}{8}} + \big\Vert \Lambda_{q_1}-\Lambda_{q_2} \big\Vert e^{2D'\rho} \Big),
\end{equation}
with $D'=D+1$. Since $\chi = 1$ on $[-1,1]$ and $\text{supp}(\chi)\subset [-2 , 2 ]$,
applying the fact that 
$$\int_\Omega (1 + \vert x_3 \vert) \big\vert q_2(x',x_3) - q_1(x',x_3) \big\vert \, dx' dx_3 \leqslant 2M < \infty ,$$
we can conclude that
$$
\Big\vert \int_{\R^3} q(x) e^{-i\xi\cdot x}  \, dx - \int_{\R^3} q(x) \chi^2\big( \rho^{-\frac{1}{4}}x_3 \big) e^{-i\xi\cdot x} \, dx  \Big\vert \leqslant \int_{\R^3} \Big\vert 1- \chi^2\big( \rho^{-\frac{1}{4}}x_3 \big)\Big\vert \vert q(x) \vert \, dx .
$$
As $1- \chi\big( \rho^{-\frac{1}{4}}x_3 \big) = 0$ for $\vert x_3\vert \leqslant \rho^{\frac{1}{4}} $ and $q_1, q_2 \in \mathcal{Q}(M,q_0)$, we have
\begin{align}\label{28}
&\Big\vert \int_{\R^3} q(x) e^{-i\xi\cdot x}  \, dx - \int_{\R^3} q(x) \chi^2\big( \rho^{-\frac{1}{4}}x_3 \big) e^{-i\xi\cdot x} \, dx  \Big\vert \nonumber \\ & \leqslant \int_{\R^2}\int_{\vert x_3 \vert \geqslant \rho^{\frac{1}{4}}} \dfrac{\vert x_3 \vert^{\frac{1}{2}}}{\rho^{\frac{1}{8}}} \vert q(x',x_3) \vert \, dx_3 dx' \nonumber \\ & \leqslant \rho^{-\frac{1}{8}} \int_{\R^3} \vert x_3 \vert^{\frac{1}{2}} \vert q(x',x_3) \vert \, dx_3 dx' \nonumber 
\\ & \leqslant \rho^{-\frac{1}{8}} \int_{\R^3} \vert x_3 \vert^{\frac{1}{2}} \vert (q_1 - q_0)(x',x_3) \vert \, dx_3 dx' + \rho^{-\frac{1}{8}} \int_{\R^3} \vert x_3 \vert^{\frac{1}{2}} \vert (q_2 - q_0)(x',x_3) \vert \, dx_3 dx' \nonumber 
\\ & \leqslant 2 M \rho^{-\frac{1}{8}}.
\end{align}
By $(\ref{27})$ and $(\ref{28})$, we get
\begin{equation}\label{29}
\vert \widehat{q}(\xi)\vert \leqslant C  \Big( R^2\rho^{-\frac{1}{8}} + \big\Vert \Lambda_{q_1}-\Lambda_{q_2} \big\Vert e^{2D'\rho} \Big).
\end{equation}
Since the constant $C$ of the above estimate depends only on $\Omega$ and $M$, we deduce that this estimate holds true for any $\xi \in \R^3$ satisfying $\eqref{cc}$. Combining this with the continuity of the map $\xi \mapsto \widehat{q}(\xi)$, which is guaranteed by the fact that $q \in L^1 (\R^3)$, we deduce that $\eqref{29}$ holds true for any $\xi \in \R^3$ satisfying $\vert \xi \vert < R$. In order to simplify the notations, we set $\gamma = \big\Vert \Lambda_{q_1}-\Lambda_{q_2} \big\Vert $. We have
\begin{align}\label{30}
\int_{\vert \xi \vert\leqslant R} \vert \widehat{q}(\xi)\vert^2 & \leqslant C R^3 \Big( R^2\rho^{-\frac{1}{8}} + \gamma e^{2D'\rho}  \Big)^2 \nonumber \\
& \leqslant 2CR^3\Big( R^4\rho^{-\frac{1}{4}} + \gamma^2 e^{4D'\rho}  \Big) 
.
\end{align}
On the other hand, as $q\in H^1(\R^3)$ and $q_1, q_2 \in \mathcal{Q}(M,q_0)$, we have 
\begin{align}\label{31}
\int_{\vert \xi \vert > R} \vert \widehat{q}(\xi)\vert^2 & \leqslant \frac{1}{R^2} \int_{\vert \xi \vert > R} \vert \xi \vert^2 \vert \widehat{q}(\xi)\vert^2  \leqslant \frac{\Vert q \Vert^2_{H^1(\R^3)}}{R^2} \leqslant \frac{2\Vert q_1 - q_0 \Vert^2_{H^1(\R^3)} + 2\Vert q_2 - q_0 \Vert^2_{H^1(\R^3)}}{R^2} \nonumber \\ & \leqslant \dfrac{4M^2}{R^2}.
\end{align}
Therefore, $(\ref{30})$ and $(\ref{31})$ imply
$$\Vert q \Vert_{L^2(\R^3)}^2 =\dfrac{\Vert \widehat{q} \Vert_{L^2(\R^3)}^2}{(2\pi)^3} \leqslant C\Big( R^7 \rho^{-\frac{1}{4}} + R^3\gamma^2 e^{4D'\rho}  + R^{-2}\Big).$$
Using the fact that $\rho > \rho_0 R^{16}$, $\rho > 1$ and $R>1$, we get $R^3 < R^{16} < \rho < e^{\rho}$. Then, we get $$R^3\gamma^2 e^{4D'\rho} \leqslant \gamma^2 e^{(4D'+1)\rho}. $$
When searching $\rho$ such that $R^7 \rho^{-\frac{1}{4}} = R^{-2}$, we find $\rho = R^{36} $.
Note that here the condition $\rho > \rho_0 R^{16}$, is still valid provided $R > 1 + \rho_o^{\frac{1}{18}}$. Then, setting $C'=4D'+1$, we get \begin{equation}\label{32}
\Vert q \Vert_{L^2(\R^3)}^2 \leqslant C\Big( R^{-2} + \gamma^2 e^{C' R^{36}} \Big).
\end{equation}
Now, let us consider the following lemma.
\begin{lemma}\label{min}
Let $a \in (0,1]$ and $b>0$. Then, there exists $C>0$ depending only on $b$ and $\rho_0$, such that $$ \underset{R>1+\rho_0^{\frac{1}{18}}}{inf} R^{-2}+a^2e^{bR^{36}} \leqslant C (\log(3+a^{-1}))^{-\frac{1}{18}}$$
\end{lemma}
This result can be deduced by choosing $R= \Big( \dfrac{2 \log \big( 3 \exp ( 1 + \rho_o^{\frac{1}{18}}) \big)  + a^{-1}}{1+b} \Big)^{\frac{1}{36}} $ and by applying standard arguments of optimization. \\
Combining \eqref{32} with Lemma \ref{min}, for $\gamma \leqslant 1$, we obtain 
\begin{equation}\label{3333}
\Vert q \Vert_{L^2(\R^3)}^2 \leqslant C (\log (3 +\gamma^{-1}))^{-\frac{1}{18}}
\end{equation}
In the same way, for $\gamma \geqslant 1$, we have 
\begin{align*}
\Vert q \Vert_{L^2(\R^3)}^2 & \leqslant 4M^2 \log (4)^{-\frac{1}{18}}(\log (3 +\gamma^{-1}))^{-\frac{1}{18}} \\
& \leqslant C (\log (3 +\gamma^{-1}))^{-\frac{1}{18}}.
\end{align*}
Combining this estimate with \eqref{3333}, we deduce that \eqref{17} holds true for $\gamma >0$. 
For $\gamma = 0 $, $\eqref{32} $ implies that $\Vert q \Vert_{L^2(\R^3)} \leqslant C R^{-1}$. Since $R>1$ is arbitrary, we can send $R$ to $+\infty$ and deduce $\eqref{17}$ for $\gamma =0$.
This completes the proof.
\end{proof}
\setcounter{equation}{0}
\section{Stability on an arbitrary part of the boundary}
In this section, inspired by the approach developed by Ben Joud \cite{5}, we will prove Theorem \ref{thm3}. Here we need to extend the arguments of \cite{5} to unbounded cylindrical domains. For this purpose, for $j = 1, 2$, we fix $q_j \in \mathcal{Q}'(M, q_0, \mathcal{O}_0)$, we consider again CGO solutions taking the form
$$
u_1(x',x_3)=e^{-\rho\theta\cdot x'}\Big(e^{i\rho\eta\cdot x}\chi\big( \rho^{-\frac{1}{4}}x_3 \big)  e^{-i\xi\cdot x} + r_1 (x) \Big) \text{ and }
u_2(x',x_3)=e^{\rho\theta\cdot x'}\Big(e^{-i\rho\eta\cdot x}\chi\big( \rho^{-\frac{1}{4}}x_3 \big) + r_2 (x) \Big),$$
with $r_1 \in H^2(\Omega)$ and  $r_2 \in H^2(\Omega)$ satisfy
$$
\rho^{-1} \Vert r_1 \Vert_{H^2(\Omega)} + \rho \Vert r_1 \Vert_{L^2(\Omega)} \leqslant C R^2 \rho^{\frac{7}{8}} \text{ and }
\rho^{-1} \Vert r_2 \Vert_{H^2(\Omega)} + \rho \Vert r_2 \Vert_{L^2(\Omega)} \leqslant C \rho^{\frac{7}{8}}.
$$
In view of Lemma \ref{lem3}, we have \begin{equation}\label{prp}
\Vert u_j \Vert_{H^2(\Omega)} \leqslant C  e^{(D+1)\rho}  ; \qquad j=1,2
\end{equation}
with $D:= \underset{x' \in \overline{\omega}}{sup} \vert x'\vert$. \\
We recall also that since $q:= q_1-q_2 = 0 $ in $\mathcal{O}_{0} $, we can extend $q$ to $H^1(\R^3)$ vector by assigning it the value $0$ outside of $\Omega$ and we denote by $q$ this extension. In this part, We need to set $\mathcal{W}_j$ ; $j=1,2,3$ such that
$$\overline{\mathcal{W}}_{j+1} \subset \mathcal{W}_j, \quad \overline{\mathcal{W}}_j \subset \mathcal{W}_0 \quad \text{and} \quad \partial \omega \subset \partial \mathcal{W}_j. $$ 
Let $\mathcal{O}_j = \mathcal{W}_j \times \R $ for $j=0,1,2,3 $.
The main idea of the proof of Theorem $\ref{thm3}$ is to combine the estimate of the Fourier transform of $q$ and the weak unique continuation property which is given in the following lemma whose proof can be found in the Appendix .
\begin{lemma}\label{UCP}
Let $M>0$, $q_1\in L^\infty (\Omega)$ such that $\Vert q \Vert_{L^\infty (\Omega)} \leqslant M$ and let $w\in H^2(\Omega)$ solve
\begin{equation}\label{eq11}
\left\lbrace
\begin{array}{l}
\text{$(-\Delta + q_1)w(x) = F(x) \quad \quad \quad \, \quad \; $ in $ \Omega$, } \\
\text{$w=0 \quad \quad \quad  \quad \quad \quad \quad \quad \quad \quad  \quad \quad $ on $ \partial \Omega, $ }
\end{array}\right.
\end{equation}
where $F \in L^2(\Omega)$. Then, there exist positive constants $C$, $\alpha_1$, $\alpha_2$ and $\lambda_0$ such that we have the following estimate:
\begin{equation}\label{eq12}
\Vert w \Vert_{H^1(\mathcal{O}_2 \backslash \mathcal{O}_3 )} \leqslant C \Big( e^{-\lambda \alpha_1} \Vert w \Vert_{H^2(\Omega)} + e^{\lambda \alpha_2} \Big( \Vert \partial_\nu w \Vert_{L^{2}(\Gamma_0 \times \R)} + \Vert F \Vert_{L^2(\mathcal{O}_0)} \Big) \Big) 
\end{equation}
for any $\lambda \geqslant \lambda_0$. Here, the constants $C$, $\alpha_1$ and $\alpha_2$ depend on $\Omega$, $M$, $\lambda_0$, $\mathcal{O}_j$ and they are independent of $q_1$, $F$, $w$ and $\lambda$.
\end{lemma}
\begin{proof}[Proof of Theorem $\ref{thm3}$]
Let $w \in H^2(\Omega)$ be the solution of
\begin{equation}\label{41}
\left\lbrace
\begin{array}{l}
\text{$-\Delta w + q_1  w= 0 \quad $ in $ \Omega$ ,} \\
\text{$w = u_2:=h \quad \; \;  \quad \; $ on $ \partial \Omega $.}
\end{array}\right.
\end{equation}
Then, $u=w-u_2$ solves 
\begin{equation}\label{42}
\left\lbrace
\begin{array}{l}
\text{$-\Delta u + q_1  u = q u_2 \quad $ in $ \Omega$ ,} \\
\text{$u = 0 \quad  \quad \quad \quad \quad \quad \; $ on $ \partial \Omega $.}
\end{array}\right.
\end{equation}
Let $\Theta$ be a cut-off function satisfying $0\leqslant \Theta \leqslant 1$, $\Theta \in \mathcal{C}^\infty (\R^2)$ and 
\begin{equation}\label{eq43}
\Theta(x')=\left\lbrace
\begin{array}{ll}
1 & \mbox{in $\omega \backslash \mathcal{W}_2 $,}\\
0 & \mbox{in $\mathcal{W}_3$.}
\end{array}
\right.
\end{equation}
We set  $$\overset{\sim }{u}(x',x_3) = \Theta (x') u(x',x_3), \quad x' \in \omega, \, x_3 \in \R .$$ We remark that $\overset{\sim }{u}$ solves 
$$\left\lbrace
\begin{array}{l}
\text{$(-\Delta + q_1)\overset{\sim }{u}(x',x_3) = \Theta(x')q(x)u_2(x)+P_1(x',D)u(x) \qquad x=(x',x_3)\in \Omega $, } \\
\text{$\overset{\sim }{u}=0 \quad \quad \quad \quad \quad \quad \quad \,\quad \quad \quad \quad \quad \quad \quad \quad  \quad \quad \qquad \qquad \qquad  \quad \;  $ on $ \partial \Omega$, } 
\end{array}\right.$$
with $P_1(x',D)$ is given by
$$ P_1(x',D)u  = [-\Delta',\Theta]u,$$ 
where $\Delta'=\partial_{x_1}^2 + \partial_{x_2}^2$. Moreover, for an arbitrary $\overset{\sim}{v} \in H^2(\Omega)$, an integration by parts leads to
$$\int_\Omega (-\Delta + q_1)\overset{\sim }{u}(x)\overset{\sim}{v}(x) \, dx = \int_\Omega (-\Delta + q_1)\overset{\sim}{v}(x) \overset{\sim }{u}(x) \, dx .$$
On the other hand, we have: 
\begin{equation}\label{44}
\int_\Omega (-\Delta + q_1)\overset{\sim }{u}(x)\overset{\sim}{v}(x) \, dx = \int_\R\int_\omega \big( \Theta(x')q(x)u_2(x)+P_1(x',D)u(x) \big)\overset{\sim}{v}(x) \, dx' \, dx_3.
\end{equation}
Choosing $\overset{\sim}{v} = u_1 $, we have $(-\Delta + q_1) \overset{\sim}{v} = 0$ in $\Omega$ and we get
\begin{equation}\label{45}
\int_\Omega  \Theta(x')q(x)u_2(x) u_1(x) \, dx = - \int_\Omega P_1(x',D)u(x)u_1(x) \, dx .
\end{equation}
Furthermore, recalling that
$$P_1(x',D)= -2 \nabla' \cdot \Theta(x') - \Delta' \Theta(x'), \quad x' \in \omega,$$
with $\nabla' = (\partial_{x_1} , \partial_{x_2} )^{T}$, we deduce that, for all $x' \in (\omega \backslash \mathcal{W}_2 ) \cup \mathcal{W}_3$, $P_1(x',D) = 0$. Thus,  $P_1(x',D)u$ is supported on $ \overline{\mathcal{O}_2} \backslash\mathcal{O}_3 $ and we find 
$$
\int_\Omega \vert P_1(x',D)u u_1 \vert \, dx  \leqslant \Vert u \Vert_{H^1(\mathcal{O}_2 \backslash\mathcal{O}_3)}\Vert u_1 \Vert_{L^2(\Omega)} \leqslant Ce^{\rho(D+1)} \Vert u \Vert_{H^1(\mathcal{O}_2 \backslash\mathcal{O}_3)}.
$$
Now, we want to make the Fourier transform of $q$ appear on the left-hand side of $(\ref{45})$. For this purpose, we use $\eqref{25}$ and the fact that $q=0$ in $\mathcal{O}_0$ to obtain 
\begin{equation} \label{46}
\int_\Omega q(x) e^{-i\xi\cdot x} \chi^2\big( \rho^{-\frac{1}{4}}x_3 \big) \, dx = - \int_\Omega P_1(x,D)u(x)u_1(x) \, dx - \int_\Omega  z_\rho \ ,dx .
\end{equation}
By $\eqref{26}$, we get
\begin{equation}\label{47}
\Big\vert \int_\Omega q(x) e^{-i\xi\cdot x} \chi^2\big( \rho^{-\frac{1}{4}}x_3 \big) \, dx \Big\vert \leqslant C \Big(  e^{\rho D'} \Vert u \Vert_{H^1(\mathcal{O}_2 \backslash \mathcal{O}_3 )} + R^2 \rho^{-\frac{1}{8}}\Big),
\end{equation}
with $D' = D+1$. In a similar way to Theorem \ref{thm2}, we obtain

$$\Big\vert  \int_{\R^3} \Big( 1 - \chi^2\big( \rho^{-\frac{1}{4}}x_3 \big)  \Big) e^{-i\xi\cdot x} q(x) \, dx  \Big\vert \leqslant \int_{\R^3} \big\vert 1 - \chi^2\big( \rho^{-\frac{1}{4}}x_3 \big)  \big\vert \vert q(x) \vert\, dx = \int_{\R^2}\int_{\vert x_3 \vert \geqslant \rho^{\frac{1}{4}}} \vert q(x',x_3) \vert \, dx_3 dx'. $$
Then, using \eqref{condnorm} in a similar way to Theorem \ref{thm2}, we can conclude that
\begin{equation}\label{48}
\Big\vert \int_{\R^3} q(x) e^{-i\xi\cdot x}  \, dx - \int_{\R^3} q(x) \chi^2\big( \rho^{-\frac{1}{4}}x_3 \big) e^{-i\xi\cdot x} \, dx  \Big\vert  \leqslant 2 M \rho^{-\frac{1}{8}}.
\end{equation}
By $(\ref{47})$ and $(\ref{48})$, we get
\begin{equation}\label{49}
\vert \widehat{q}(\xi)\vert \leqslant C  \Big( R^2\rho^{-\frac{1}{8}} + e^{\rho D'} \Vert u \Vert_{H^1(\mathcal{O}_2 \backslash \mathcal{O}_3 )} \Big).
\end{equation}
Now, we have just to combine $\eqref{49}$ and $\eqref{eq12}$ (see Appendix A) to get
\begin{equation}\label{fr}
\vert \widehat{q}(\xi)\vert \leqslant C \Big[ R^2 \rho^{-\frac{1}{8}}+  e^{D'\rho} \Big( e^{-\lambda \alpha_1} \Vert u \Vert_{H^2(\Omega)} + e^{\lambda \alpha_2} \big( \Vert \partial_\nu u \Vert_{L^{2}(\Gamma_0 \times \R)} + \Vert F \Vert_{L^2(\mathcal{O}_0)} \big) \Big) \Big] ,
\end{equation}
with $\lambda > \lambda_0 > 0$ arbitrary chosen and $F = - \Delta u + q_1 u$. Here $\lambda_0$ is the constant appearing in Lemma \ref{UCP} of the Appendix A. In order to simplify the notations, we set $\gamma_1= \Vert \Lambda'_{q_1}-\Lambda'_{q_2} \Vert_{\mathcal{B}(H^{\frac{3}{2}} (\partial \Omega) , H^{\frac{1}{2}} (\Gamma_0 \times R ))} $.
Since $ \partial_\nu w = (\Lambda_{q_1}-\Lambda_{q_2}) (h)$, where $h$ is given by $(\ref{41})$, we have

\begin{align*}
\Vert \partial_\nu u \Vert_{L^{2}(\Gamma_0 \times \R)} & \leqslant C \Vert \Lambda'_{q_1}-\Lambda'_{q_2} \Vert_{\mathcal{B}(H^{\frac{3}{2}} (\partial \Omega) , H^{\frac{1}{2}} (\Gamma_0 \times R ))} \Vert h\Vert_{H^{\frac{3}{2}}(\partial \Omega)} \\
&  \leqslant C e^{D' \rho } \Vert \Lambda'_{q_1}-\Lambda'_{q_2} \Vert_{\mathcal{B}(H^{\frac{3}{2}} (\partial \Omega) , H^{\frac{1}{2}} (\Gamma_0 \times R ))} .
\end{align*}
Moreover, since $q_j \in \mathcal{Q}'(M, q_0 , \mathcal{O}_0)$, $j=1,2$, we have $q_1 = q_0 = q_2$ on $\mathcal{O}_0$. Therefore, we have $F= (q_1 - q_2) u_2 = 0$ on $\mathcal{O}_0$ and it follows
\begin{equation}\label{50}
\vert \widehat{q}(\xi)\vert \leqslant C \Big[ R^2 \rho^{-\frac{1}{8}}+  e^{D'\rho} \Big( e^{-\lambda \alpha_1} \Vert u \Vert_{H^2(\Omega)} + e^{\lambda \alpha_2 + D' \rho }  \gamma_1\Big) \Big] .
\end{equation}
By $\eqref{prp}$, we get
\begin{equation}\label{51}
\vert \widehat{q}(\xi)\vert \leqslant C \Big( R^2 \rho^{-\frac{1}{8}}+  e^{2D'\rho-\lambda \alpha_1} + e^{2D'\rho+\lambda \alpha_2}  \gamma_1  \Big) ,
\end{equation}
for all $\xi \in \R^3$ such that $\vert \xi \vert \leqslant R$. Let $\lambda = \tau \rho $. Choosing $\tau $ sufficiently large, it becomes easy to find constants $\alpha_3$ and $\alpha_4$ such that 
\begin{equation}\label{52}
e^{2D'\rho-\lambda \alpha_1} =  e^{\rho(2D' -\tau \alpha_1)} \leqslant e^{-\alpha_3 \rho} \quad \text{ and } \quad e^{2D'\rho+\lambda \alpha_2}=e^{\rho(2D'  + \tau \alpha_2)} \leqslant e^{\alpha_4 \rho} .
\end{equation}
Combining $\eqref{51}$ and $\eqref{52}$, we conclude that, for any $\rho \geqslant \rho_0 R^{16}$, we have
\begin{align}\label{53}
\vert \widehat{q}(\xi)\vert & \leqslant C \Big( R^2 \rho^{-\frac{1}{8}}+  e^{- \alpha_3 \rho} + e^{\alpha_4 \rho }  \gamma_1 \Big) \nonumber \\
& \leqslant C \Big( R^2 \rho^{-\frac{1}{8}} + e^{\alpha_4 \rho }  \gamma_1  \Big) .
\end{align} 
It follows \begin{align}\label{54}
\int_{\vert \xi \vert\leqslant R} \vert \widehat{q}(\xi)\vert^2 & \leqslant C R^3 \Big( R^2 \rho^{-\frac{1}{8}} + e^{\alpha_4 \rho }  \gamma_1 \Big)^2 \nonumber \\
& \leqslant 2CR^3\Big( R^4\rho^{-\frac{1}{4}} + \gamma_1^2 e^{2\alpha_4\rho}  \Big) 
.
\end{align}
On the other hand, as $q\in H^1(\R^3)$, we have 
\begin{equation}\label{55}
\int_{\vert \xi \vert > R} \vert \widehat{q}(\xi)\vert^2 \leqslant \frac{1}{R^2} \int_{\vert \xi \vert > R} \vert \xi \vert^2 \vert \widehat{q}(\xi)\vert^2 \leqslant \frac{\Vert q \Vert^2_{H^1(\R^3)}}{R^2}\leqslant \dfrac{M^2}{R^2}.
\end{equation}
Then, $\eqref{54}$ and $\eqref{55}$ imply 
$$\Vert q \Vert_{L^2(\R^3)}^2 = \dfrac{\Vert  \widehat{q} \Vert_{L^2(\R^3)}^2}{(2\pi)^3} \leqslant C\Big( R^7 \rho^{-\frac{1}{4}} + R^3\gamma_1^2 e^{2\alpha_4\rho}  + R^{-2}\Big).$$
Since $\rho \geqslant \rho_0 R^{16}$, we have $R^3 \leqslant e^{\rho}$. So that, we get $$R^3\gamma_1^2 e^{2\alpha_4\rho} \leqslant \gamma_1^2 e^{(2\alpha_4+1)\rho}. $$
When searching $\rho$ such that $R^7 \rho^{-\frac{1}{4}} = R^{-2}$, we find $\rho = R^{36} $. Then, setting $C'=2\alpha_4+1$, we get \begin{equation}\label{56}
\Vert q \Vert_{L^2(\R^3)}^2 \leqslant C\Big( R^{-2} + \gamma_1^2 e^{C' R^{36}} \Big).
\end{equation}
Then, by using Lemma \ref{min} and repeating the arguments used at the end of the proof of Theorem \ref{thm2}, we can deduce $\eqref{555}$ from $\eqref{56}$.
\end{proof}
\setcounter{equation}{0}
\appendix
\section{ }
\subsection{Carleman's estimate}
Inspired by the work of Kian, Sambo and Soccorsi \cite{28} and Bellassoued, Kian and Soccorsi \cite{3}, we prove here a Carleman estimate for the Schr\"odinger operator $-\Delta + q$ in an infinite cylindrical domain in order to localize the observation and
derive the estimate \eqref{fr} which is a key ingredient in the proof of Theorem \ref{thm3}. As known, the Carleman's estimates are weighted inequalities. So, we need to build a weight function with particular properties. The existence of such a function is guaranteed by the following lemma borrowed from \cite[Lemma 2.3]{21} (see also \cite[Lemma 1.2]{19} and \cite[Theorem 2.4]{32}). 
\begin{lemma}
There exists a function $\psi_0 \in \mathcal{C}^3(\overline{\mathcal{W}}_0)$ such that:
\begin{enumerate}[label=(\roman*)]
\item $\psi_0 (x') > 0$ for all $x' \in \mathcal{W}_0$,
\item There exists $\alpha_0 > 0$ such that $\vert \nabla^{'} \psi_0 (x') \vert \geqslant \alpha_0 $ for all $x' \in \overline{\mathcal{W}}_0$,
\item $\partial_{\nu'} \psi_0 (x') \leqslant 0 $ for all $x' \in \partial \mathcal{W}_0 \backslash \Gamma_0$,
\item $\psi_0 (x')= 0$ for all $x' \in \partial \mathcal{W}_0 \backslash \Gamma_0$.
\end{enumerate}
\end{lemma}
Here $\nabla^{'}$ denotes the gradient with respect to $x'\in \R^2$ and $\partial_{\nu'}$ is the normal derivative with respect to $\partial \mathcal{W}_0$, that is $\partial_{\nu'} := \nu' \cdot \nabla^{'}$ where $\nu'$ stands for the outward normal vector to $\partial  \mathcal{W}_0$. Note that the last condition (iv) can be deduced from the construction of the weight function $\psi_0$ in the proof of \cite[Lemma 2.3]{21} combined with properties borrowed from \cite[Lemma 2.1]{21}.\\
Thus, putting $\psi (x) = \psi(x',x_3):=\psi_0(x')$ for all $x=(x',x_3)\in \overline{\mathcal{O}_0}$, it is apparent that the function $\psi \in \mathcal{C}^3(\overline{\mathcal{O}_0})$ satisfies the three following conditions:
\begin{itemize}
\item[(C1)] $\psi(x) > 0$, \quad $x \in \mathcal{O}_0$,
\item[(C2)] $\vert \nabla \psi (x) \vert \geqslant \alpha_0 $ for all $x \in \overline{\mathcal{O}_0}$,
\item[(C3)] $\partial_{\nu} \psi(x) \leqslant 0 $ for all $x \in \partial \mathcal{O}_0 \backslash (\Gamma_0 \times \R)$,
\item[(C4)] $\psi(x)=0$ for all $x \in \Gamma^{\sharp} \times \R$.
\end{itemize}
Here $\nu$ is the outward unit normal vector to the boundary $\partial \mathcal{O}_0$. Evidently $\nu=(\nu',0)$ so we have $ \partial_{\nu} \psi = \partial_{\nu'} \psi_0 $ as the function $\psi$ does not depend on $x_3$.\\
Next, for $\beta \in (0,+\infty)$, we introduce the following weigh function 
\begin{equation}\label{eq1}
\varphi(x)=\varphi(x')=e^{\beta \psi(x)}; \quad x \in \mathcal{O}_0 .
\end{equation}
Through the following Lemma, we introduce some properties of $\varphi$ that will be used after.
\begin{lemma}\label{phi}
There exists a constant $\beta_0 \in (0, +\infty)$ depending only on $\psi$ such that the following statements hold uniformly in $\mathcal{O}_0$ for all $\beta \in [\beta_0 , +\infty )$.
\begin{enumerate}[label=(\alph*)]
\item $\vert \nabla \varphi \vert \geqslant \alpha:= \beta_0 \alpha_0 $,
\item $\nabla \vert \nabla \varphi \vert^2 \cdot \nabla \varphi \geqslant C_0 \beta \vert \nabla \varphi \vert^3$,
\item $\mathcal{H}(\varphi)\xi \cdot \xi + C_1 \beta \vert \nabla \varphi \vert \vert \xi \vert^2 \geqslant 0 \quad ; \quad \xi \in \R^3$,
\item $\vert \Delta \vert \nabla \varphi \vert \vert \leqslant C_2 \vert \nabla \varphi \vert^3$,
\item $\Delta \varphi \geqslant 0$.
\end{enumerate}
Here, $C_0$, $C_1$ and $C_2$ are positive constants depending only on $\psi$ and $\alpha_0$ and $\mathcal{H}(\varphi)$ denotes the Hessian matrix of $\varphi$ with respect to $x\in \mathcal{O}_0$.
\end{lemma}
\begin{proof}~~
\begin{enumerate}[label=(\alph*)]
\item As we have \begin{equation}\label{eq9}
\nabla \varphi = \beta \nabla \psi e^{\beta \psi},
\end{equation}
we can directly deduce $(a)$ from $(C2)$
\item With reference to $(\ref{eq1})$ and $(\ref{eq9})$, we see that 
$$\nabla \vert \nabla \varphi \vert = \beta \Big( \beta  e^{\beta \psi} \vert \nabla \psi \vert \nabla \psi + e^{\beta \psi} \nabla \vert \nabla \psi \vert \Big) = \beta \vert \nabla \psi \vert \nabla \varphi +  \dfrac{\vert \nabla \varphi \vert}{\vert \nabla \psi \vert}\nabla \vert \nabla \psi \vert,  $$
and hence \begin{equation}\label{eq10}
\nabla \vert \nabla \varphi \vert^2 \cdot  \nabla \varphi  = 2 \Big( \beta \vert \nabla \psi \vert \vert \nabla \varphi \vert^3 + \vert \nabla \varphi \vert^2 \dfrac{\nabla \vert \nabla \psi \vert \cdot \nabla \varphi}{\vert \nabla \psi \vert}  \Big).
\end{equation}
By $(C_2)$, we have $\dfrac{\nabla \vert \nabla \psi \vert \cdot \nabla \varphi}{\vert \nabla \psi \vert} \leqslant C \alpha_0^{-1} \vert \nabla \varphi \vert$, where $C$ is a positive constant depending only on $\psi$. Therefore, $(\ref{eq10})$ yields 
$$\nabla \vert \nabla \varphi \vert^2 \cdot  \nabla \varphi  \geqslant 2 \beta \alpha_0 \big( 1-\dfrac{C}{\beta \alpha_0^2} \big)\vert \nabla \varphi \vert^3 .$$
Finally, by taking $\beta_0 \in [2 \alpha_0^2 , + \infty )$, we get $(b)$.
\item For all $i$ and $j \in \lbrace 1,2,3 \rbrace$, we have 
$$\partial_{x_i} \partial_{x_j} \varphi = \beta e^{\beta \psi} \big( \partial_{x_i} \partial_{x_j} \psi + \beta (\partial_{x_i} \psi ) (\partial_{x_j} \psi ) \big) = \big( \partial_{x_i} \partial_{x_j} \psi + \beta (\partial_{x_i} \psi ) (\partial_{x_j} \psi ) \big) \dfrac{\vert \nabla \varphi \vert}{\vert \nabla \psi \vert}. $$
Thus, using $(C2)$, each $\vert \partial_{x_i} \partial_{x_j} \varphi \vert$, $i,j=1,2,3$, is upper bounded by $ C \beta \vert \nabla \varphi \vert$, where $C$ is a positive constant depending only on $\psi$ and $\alpha_0$. As a consequence, there exists $C' = C'(\psi , \alpha_0) \in (0 , + \infty)$, such that 
$$\vert \mathcal{H}(\varphi) \xi \cdot \xi \vert \leqslant C' \beta \vert \nabla \varphi \vert \vert \xi \vert^2 \quad ; \quad \xi \in \R^3.$$ 
\item We have $$  \Delta \vert \nabla \varphi \vert  = \beta e^{\beta \psi} \Big( \beta^2 \vert \nabla \psi \vert^3  + \beta (\Delta \psi )\vert \nabla \psi \vert + 2 \beta  \nabla \psi \cdot \nabla \vert \nabla \psi \vert + \Delta \vert \nabla \psi \vert \Big) $$
As we have $\vert \nabla \varphi\vert = \beta e^{\beta \psi} \vert \nabla \psi \vert$ and by $(C2)$, we get 
$$\big\vert \Delta \vert \nabla \varphi \vert \big\vert \leqslant C(\psi) \beta^3 e^{\beta \psi} \leqslant C(\psi) \vert \nabla \varphi\vert^3. $$
\item We have $\Delta \varphi  = e^{\beta \psi} \big( \beta^2 \vert \nabla \psi \vert^2 + \beta \Delta \psi  \big)$. By $(a)$ in Lemma \ref{phi}, we get 
$$ \Delta \varphi \geqslant \alpha_0^2 \beta^2 + \beta \Delta \psi \geqslant \alpha_0^2 \beta^2 - \beta \Vert \psi \Vert_{W^{2,\infty}(\mathcal{O}_0)}. $$ For $\beta \geqslant \dfrac{ \Vert \psi \Vert_{W^{2,\infty}(\mathcal{O}_0)}}{\alpha_0^2} $, we get $\Delta \varphi \geqslant 0$.
\end{enumerate}
\end{proof}
Now, we may state the following Carleman's estimate for the operator $-\Delta + q$.
\begin{theorem}
Let $u \in H_0^1(\mathcal{O}_0) \cap H^2(\mathcal{O}_0)$, $M>0$ and let $q \in L^{\infty}(\Omega)$ satisfy $\Vert q  \Vert_{L^{\infty} (\mathcal{O}_0)} \leqslant M$. Then, there exists $\beta_0 \in (0, +\infty)$ such that for every $\beta \geqslant \beta_0$, there is $\lambda_0=\lambda_0(\beta) \in (0, +\infty)$ depending only on $\beta$, $\alpha_0$, $\mathcal{O}_0$, $M$ and $\Gamma_0$, such that the estimate 
\begin{equation}\label{eq2}
\lambda \int_{\mathcal{O}_0} e^{2\lambda\varphi} \big( \lambda^2 \vert u \vert^2 + \vert \nabla u \vert^2 \big) \, dx \\
\leqslant C \Big( \int_{\mathcal{O}_0} e^{2\lambda\varphi} \vert (-\Delta + q)u \vert^2 \, dx + \lambda \int_{\Gamma_0 \times \R} e^{2\lambda\varphi} \big\vert \partial_\nu u \big\vert^2  \, d\sigma_x \Big).
\end{equation}
holds for all $\lambda \geqslant \lambda_0$ and some positive constant $C$ that depends only on $\alpha_0$, $\omega$, $\Gamma_0$, $\beta$ and $\lambda_0$.
\end{theorem}
\begin{proof}
For the proof, as $ \Vert q \Vert _{L^\infty (\mathcal{O}_0 )} \leqslant M $, we can simply show the following inequality 
\begin{equation}\label{eq3}
\lambda \int_{\mathcal{O}_0} e^{2\lambda\varphi} \big( \lambda^2 \vert u \vert^2 + \vert \nabla u \vert^2 \big) \, dx \\
\leqslant C \Big( \int_{\mathcal{O}_0} e^{2\lambda\varphi} \vert \Delta u \vert^2 \, dx + \lambda \int_{\Gamma_0 \times \R } e^{2\lambda\varphi} \big\vert \partial_\nu u \big\vert^2  \, d\sigma_x \Big).
\end{equation}
Without loss of generality we may assume that $u$ is real valued. We set $v(x)=e^{\lambda\varphi(x)} u(x)$ in such a way that $$\int_{\mathcal{O}_0} e^{2\lambda\varphi} \vert \Delta u \vert^2 \, dx = \int_{\mathcal{O}_0} \vert- e^{\lambda\varphi} \Delta(e^{-\lambda\varphi} v ) \vert^2 \, dx  .$$ We have 
$$ \Delta (e^{-\lambda\varphi} v) = e^{-\lambda\varphi} ( \lambda^2 \vert\nabla\varphi\vert^2 v - 2\lambda \nabla\varphi\nabla v - \lambda v \Delta\varphi +\Delta v).$$
Then, \begin{align*}
P_\lambda v & := - e^{\lambda\varphi} \Delta(e^{-\lambda\varphi} v )= - \Delta v + \lambda v \Delta \varphi + 2 \lambda \nabla \varphi \nabla v - \lambda^2 \vert \nabla \varphi \vert^2 v \\
& = P_\lambda^-v(x)+P_\lambda^+v(x) + Rv(x),
\end{align*}
with $P_\lambda^-v= 2 \lambda \nabla \varphi \cdot \nabla v $, $P_\lambda^+v= - \Delta v - \lambda^2 \vert \nabla \varphi \vert^2 v$ and $ Rv=\lambda v \Delta \varphi  $. Let $\overset{\sim}{P}_\lambda v = P_\lambda - Rv = P_\lambda^-v+P_\lambda^+v$.  
With the previous notations, we have $\Vert \overset{\sim}{P}_\lambda v \Vert^2 = \Vert P_\lambda^+v \Vert_{L^2(\mathcal{O}_0)}^2 + \Vert P_\lambda^-v \Vert_{L^2(\mathcal{O}_0)}^2 + 2 Re\langle P_\lambda^+v , P_\lambda^-v \rangle_{L^2(\mathcal{O}_0)}$. Then, as $v$ is real valued, we compute the term $2 \langle P_\lambda^+v , P_\lambda^-v \rangle_{L^2(\mathcal{O}_0)}$. We have
\begin{align*}
\int_{\mathcal{O}_0} P_\lambda^+v P_\lambda^-v & = - 2 \lambda \int_{\mathcal{O}_0} \Delta v \nabla \varphi \cdot \nabla v \, dx - 2 \lambda^3 \int_{\mathcal{O}_0} \vert \nabla \varphi \vert^2 \nabla \varphi \cdot \nabla v \, v \, dx, \\
& := K_1 + K_2.
\end{align*}
$$K_1 = 2 \lambda \int_{\mathcal{O}_0} \nabla v \cdot \nabla (\nabla \varphi \cdot \nabla v) \, dx - 2 \lambda \int_{\partial \mathcal{O}_0} \partial_\nu v \, \nabla \varphi \cdot \nabla v \, d\sigma_x := J_1 + J_2. $$
However
$$
\nabla v \cdot  \nabla(\nabla\varphi \cdot \nabla v ) = \mathcal{H}(\varphi) (\nabla v , \nabla v ) + \frac{1}{2} \nabla \varphi \cdot \nabla ( \vert \nabla v \vert^2) . $$
Then \begin{align*}
J_1 & = 2 \lambda \int_{\mathcal{O}_0} \mathcal{H}(\varphi) (\nabla v , \nabla v ) \, dx + \lambda \int_{\mathcal{O}_0} \nabla \varphi \cdot \nabla ( \vert \nabla v \vert^2) \, dx \\ 
& = 2 \lambda \int_{\mathcal{O}_0} \mathcal{H}(\varphi) (\nabla v , \nabla v ) \, dx - \lambda \int_{\mathcal{O}_0} \Delta \varphi  \vert \nabla v \vert^2 \, dx + \lambda \int_{\partial \mathcal{O}_0} \nabla \varphi \cdot \nu  \vert \nabla v \vert^2 \, d\sigma_x.
\end{align*} 
As a result, we get 
$$K_1 = 2 \lambda \int_{\mathcal{O}_0} \mathcal{H}(\varphi) (\nabla v , \nabla v ) \, dx - \lambda \int_{\mathcal{O}_0} \Delta \varphi  \vert \nabla v \vert^2 \, dx + \lambda \int_{\partial \mathcal{O}_0} \nabla \varphi \cdot \nu  \vert \nabla v \vert^2 \, d\sigma_x - 2 \lambda \int_{\partial \mathcal{O}_0} \partial_\nu v \, \nabla \varphi \cdot \nabla v \, d\sigma_x. $$
As we have $v=0$ on $\partial \mathcal{O}_0$ and $\nabla v = ( \partial_\nu v ) \nu$ on $\partial \mathcal{O}_0$, we get
$$K_1 = 2 \lambda \int_{\mathcal{O}_0} \mathcal{H}(\varphi) (\nabla v , \nabla v ) \, dx  - \lambda \int_{\mathcal{O}_0} \Delta \varphi  \vert \nabla v \vert^2 \, dx - \lambda \int_{\partial \mathcal{O}_0} \partial_\nu \varphi \vert \partial_\nu v \vert^2 \, d\sigma_x. $$
$$
K_2 = \lambda^3 \int_{\mathcal{O}_0} \vert v \vert^2\big( \vert \nabla \varphi \vert^2 \Delta \varphi + \nabla \varphi \cdot \nabla (\vert \nabla \varphi \vert^2)  \big) \, dx.
$$
Thus, we obtain:
\begin{multline*}
\int_{\mathcal{O}_0} P_\lambda^+v P_\lambda^-v \, dx+ \lambda \int_{\partial \mathcal{O}_0} \partial_\nu \varphi  \vert \partial_\nu v \vert^2 \, d\sigma_x \\ = 2 \lambda \int_{\mathcal{O}_0} \mathcal{H}(\varphi)  (\nabla v , \nabla v ) \, dx - \lambda \int_{\mathcal{O}_0} \Delta \varphi  \vert \nabla v \vert^2 \, dx + \lambda^3 \int_{\mathcal{O}_0} v^2\big( \vert \nabla \varphi \vert^2 \Delta \varphi + \nabla \varphi \cdot \nabla (\vert \nabla \varphi \vert^2)  \big) \, dx.
\end{multline*}
Referring to $(b)$ and $(c)$ in Lemma \ref{phi} and condition (C3) which implies that $ \partial_\nu \varphi (x) \leqslant 0$ for $ x \in \partial \mathcal{O}_0 \backslash (\Gamma_0 \times \R)$, we obtain for all $\beta \in [\beta_0 , + \infty )$, that:
\begin{equation}\label{eq4}
\int_{\mathcal{O}_0} P_\lambda^+v P_\lambda^-v + \lambda \int_{\Gamma_0\times \R} \partial_\nu \varphi \vert \partial_\nu v \vert^2 \, d\sigma_x \geqslant C_0 \lambda^3 \beta \int_{\mathcal{O}_0}\vert \nabla \varphi \vert^3 \vert v\vert^2 \, dx - 2 C_1 \lambda \beta \int_{\mathcal{O}_0} \vert \nabla \varphi \vert \vert \nabla v \vert^2 \, dx + R_0,
\end{equation} 
with $ R_0= \displaystyle\int_{\mathcal{O}_0} \lambda \Delta \varphi \big( \lambda^2 \vert v \vert^2 \vert \nabla \varphi \vert^2 - \vert \nabla v \vert^2  \big) \, dx $. As we have $P_\lambda^+v= - \Delta v - \lambda^2 \vert \nabla \varphi \vert^2 v$, we can write $\Delta v = - P_\lambda^+v - \lambda^2 \vert \nabla \varphi \vert^2 v$, and then
$$\vert \Delta v \vert^2 = ( P_\lambda^+v + \lambda^2 \vert \nabla \varphi \vert^2 v)^2 \leqslant 3 \big( (P_\lambda^+v )^2 + \lambda^4 \vert \nabla \varphi \vert^4 \vert v \vert^2 \big). $$
From $(a)$ in Lemma \ref{phi}, we get for all $\beta \in [\beta_0 , + \infty )$
$$\big( \lambda \vert \nabla \varphi \vert \big)^{-1}\vert \Delta v \vert^2 \leqslant \dfrac{3}{\alpha \lambda} ( P_\lambda^+v )^2 + 3\lambda^3 \vert \nabla \varphi \vert^3 \vert v \vert^2 . $$ This entails that 
\begin{equation}\label{eq5}
\int_{\mathcal{O}_0} \big( \lambda \vert \nabla \varphi \vert \big)^{-1} \vert \Delta v \vert^2 \, dx \leqslant \dfrac{3}{\alpha \lambda}\int_{\mathcal{O}_0} \vert P_\lambda^+v \vert^2 \, dx+ 3\int_{\mathcal{O}_0} (\lambda \vert \nabla \varphi \vert)^3 \vert v \vert^2 \, dx .
\end{equation}
Further, as $v=0$ on $\partial \mathcal{O}_0$, we get
$$\int_{\mathcal{O}_0} \vert \nabla \varphi \vert \vert \nabla v \vert^2 \, dx = - \int_{\mathcal{O}_0} \vert \nabla \varphi \vert v \Delta v dx + \frac{1}{2} \int_{\mathcal{O}_0} \Delta \vert \nabla \varphi \vert \vert v\vert^2 \, dx .$$
This and the estimate 
$$ \beta^{\frac{3}{2}} \lambda \vert \nabla \varphi \vert \vert v \Delta v \vert = \big( (\lambda \vert \nabla \varphi \vert)^{-\frac{1}{2}} \vert \Delta v \vert \big) \big( \beta^{\frac{3}{2}} (\lambda \vert \nabla \varphi \vert)^{\frac{3}{2}} \vert v \vert  \big) \leqslant \dfrac{1}{2} \big( \lambda \vert \nabla \varphi \vert \big)^{-1}\vert \Delta v \vert^2 + \dfrac{\beta^3}{2} \big(\lambda \vert \nabla \varphi \vert\big)^3 \vert v \vert^2  $$ 
lead to 
\begin{equation}\label{eq6}
\beta^{\frac{3}{2}} \lambda \int_{\mathcal{O}_0} \vert \nabla \varphi \vert \vert \nabla v \vert^2 \, dx \leqslant \dfrac{1}{2} \int_{\mathcal{O}_0} \big( \lambda \vert \nabla \varphi \vert \big)^{-1}\vert \Delta v \vert^2 \, dx + \dfrac{\beta^3}{2} \int_{\mathcal{O}_0} \big(\lambda \vert \nabla \varphi \vert\big)^3 \vert v \vert^2  \, dx + \beta^{\frac{3}{2}} \lambda \int_{\mathcal{O}_0} \Delta \vert \nabla \varphi \vert \vert v \vert^2 \, dx .
\end{equation}
From $(\ref{eq5})$ and $(\ref{eq6})$ it follows that
\begin{multline*}
\beta^{\frac{3}{2}} \lambda \int_{\mathcal{O}_0} \vert \nabla \varphi \vert \vert \nabla v \vert^2 \, dx + \dfrac{1}{2} \int_{\mathcal{O}_0} \big( \lambda \vert \nabla \varphi \vert \big)^{-1}\vert \Delta v \vert^2 \, dx \\
\leqslant \dfrac{3}{\alpha \lambda}\int_{{\mathcal{O}_0}} \vert P_\lambda^+v \vert^2 \, dx + \big( 3 + \dfrac{\beta^3}{2} \big) \int_{{\mathcal{O}_0}} (\lambda \vert \nabla \varphi \vert)^3 \vert v \vert^2 \, dx + R_1,
\end{multline*}
where $R_1 = \dfrac{\beta^{\frac{3}{2}} }{2}\lambda \displaystyle \int_{\mathcal{O}_0} \Delta \vert \nabla \varphi \vert \vert v \vert^2 \, dx$. Therefore, upon substituting $\text{max}(\beta_0 , 6^{\frac{1}{3}})$ for $\beta_0$, we obtain for all $\beta \in [\beta_0 , + \infty )$ that
$$
\beta^{\frac{3}{2}} \lambda \int_{\mathcal{O}_0} \vert \nabla \varphi \vert \vert \nabla v \vert^2 \, dx + \dfrac{1}{2} \int_{\mathcal{O}_0} \big( \lambda \vert \nabla \varphi \vert \big)^{-1} \vert \Delta v \vert^2 \, dx 
\leqslant \dfrac{3}{\alpha \lambda}\int_{{\mathcal{O}_0}} \vert P_\lambda^+v \vert^2 \, dx + \beta^3 \int_{{\mathcal{O}_0}} (\lambda \vert \nabla \varphi \vert)^3 \vert v \vert^2 \, dx + R_1.
$$
Putting this together with $(\ref{eq4})$, we find
\begin{multline}\label{eq7}
\dfrac{3}{\alpha \lambda}\int_{{\mathcal{O}_0}} \vert P_\lambda^+v \vert^2 \, dx + \dfrac{2}{C_0} \int_{{\mathcal{O}_0}} P_\lambda^+v P_\lambda^-v \, dx + \dfrac{2\lambda}{C_0} \int_{\Gamma_0 \times \R}  \partial_\nu \varphi \vert \partial_\nu v \vert^2 \, d\sigma_x \\
\geqslant \big( \beta^{\frac{1}{2}} - \dfrac{4C_1}{C_0}  \big) \beta\lambda \int_{\mathcal{O}_0} \vert \nabla \varphi \vert \vert \nabla v \vert^2 \, dx + \beta \int_{{\mathcal{O}_0}} ( \lambda \vert \nabla \varphi \vert)^3 \vert v \vert^2 \, dx 
+ \dfrac{1}{2} \int_{\mathcal{O}_0} \big( \lambda \vert \nabla \varphi \vert \big)^{-1}\vert \Delta v \vert^2 \, dx   + R_2,
\end{multline}
with $R_2 = \dfrac{2R_0}{C_0} - R_1$.
Then enlarging $\beta_0$ in such a way that $\beta_0^{\frac{1}{2}} - \dfrac{4C_1}{C_0} $ is lower bounded by $C_2 >0$, we infer from $(\ref{eq7})$ that for all $\beta \in [\beta_0 , + \infty )$
\begin{multline*}
\dfrac{3}{\alpha \lambda}\int_{{\mathcal{O}_0}} ( P_\lambda^+v )^2 \, dx + \dfrac{2}{C_0} \int_{{\mathcal{O}_0}} P_\lambda^+v P_\lambda^-v \, dx + \dfrac{2\lambda}{C_0} \int_{\Gamma_0 \times \R} \vert \partial_\nu \varphi \vert \vert \partial_\nu v \vert^2 \, d\sigma_x \\
\geqslant C_2 \lambda \int_{\mathcal{O}_0} \vert \nabla \varphi \vert \vert \nabla v \vert^2 \, dx + \beta \int_{{\mathcal{O}_0}} ( \lambda \vert \nabla \varphi \vert)^3 \vert v \vert^2 \, dx 
+ \dfrac{1}{2} \int_{\mathcal{O}_0} \big( \lambda \vert \nabla \varphi \vert \big)^{-1}\vert \Delta v\vert^2 \, dx   + R_2.
\end{multline*}
Further, since $P_\lambda^-v= 2 \lambda \nabla \varphi \nabla v $, by $(a)$ in Lemma \ref{phi}, we get 
\begin{multline*}
\dfrac{3}{\alpha \lambda}\int_{{\mathcal{O}_0}} \vert P_\lambda^+v \vert^2 \, dx + \dfrac{2}{C_0} \int_{{\mathcal{O}_0}} P_\lambda^+v P_\lambda^-v \, dx +\dfrac{1}{4 \alpha^2 \lambda^2} \int_{{\mathcal{O}_0}} \vert P_\lambda^-v \vert^2 \, dx  +\dfrac{2\lambda}{C_0} \int_{\Gamma_0} \partial_\nu \varphi \vert \partial_\nu v \vert^2 \, d\sigma_x \\
\geqslant C_2 \lambda \int_{\mathcal{O}_0} \vert \nabla \varphi \vert \vert \nabla v \vert^2 \, dx + \int_{\mathcal{O}_0} \vert \nabla v \vert^2 \, dx+ \beta \int_{{\mathcal{O}_0}} ( \lambda \vert \nabla \varphi \vert)^3 \vert v \vert^2 \, dx 
+ \dfrac{1}{2} \int_{\mathcal{O}_0} \big( \lambda \vert \nabla \varphi \vert \big)^{-1} \vert \Delta v \vert^2 \, dx   + R_2.
\end{multline*}
Therefore, bearing in mind that $\overset{\sim}{P}_\lambda =  P_\lambda^+ +  P_\lambda^-$, we get that 
\begin{multline*}
\dfrac{1}{C_0} \Big( \int_{{\mathcal{O}_0}} \vert \overset{\sim}{P}_\lambda v \vert^2 \, dx  + 2 \lambda \int_{\Gamma_0}  \partial_\nu \varphi  \vert \partial_\nu v \vert^2 \, d\sigma_x \Big) \\
\geqslant \dfrac{1}{2}  \int_{\mathcal{O}_0} \big( \lambda \vert \nabla \varphi \vert \big)^{-1}\vert \Delta v \vert^2 \, dx +C_2 \lambda \int_{\mathcal{O}_0} \vert \nabla \varphi \vert \vert \nabla v \vert^2 \, dx + \beta \int_{{\mathcal{O}_0}} ( \lambda \vert \nabla \varphi \vert)^3 \vert v \vert^2 \, dx  + R_2,
\end{multline*}
provided $\lambda \in \Big[ \dfrac{4C_0}{\alpha} , + \infty \Big) $. We have 
$$
R_2  = \dfrac{2}{C_0} R_0 - R_1 = \dfrac{2}{C_0} \int_{\mathcal{O}_0} \lambda \Delta \varphi \big( \lambda^2 \vert v \vert^2 \vert \nabla \varphi \vert^2 - \vert \nabla v \vert^2  \big) \, dx - \dfrac{\beta^{\frac{3}{2}} }{2}\lambda \displaystyle \int_{\mathcal{O}_0} \Delta \vert \nabla \varphi \vert \vert v \vert^2 \, dx .
$$
By $(d)$ and $(e)$ in Lemma \ref{phi} and for $\beta >1$ sufficiently large, we get $\vert R_2 \vert \leqslant C_2  \displaystyle\int_{{\mathcal{O}_0}} ( \lambda \vert \nabla \varphi \vert)^3 v^2 \, dx $. Then 
\begin{multline}\label{eq8}
\dfrac{1}{C_0} \Big( \int_{{\mathcal{O}_0}} \vert \overset{\sim}{P}_\lambda v \vert^2 \, dx  + 2 \lambda \int_{\Gamma_0 \times \R} \partial_\nu \varphi  \vert \partial_\nu v \vert^2 \, d\sigma_x \Big) \\
\geqslant \dfrac{1}{2}  \int_{\mathcal{O}_0} \big( \lambda \vert \nabla \varphi \vert \big)^{-1}\vert \Delta v \vert^2 \, dx +C_2 \lambda \int_{\mathcal{O}_0} \vert \nabla \varphi \vert \vert \nabla v \vert^2 \, dx + C_2 \beta \int_{{\mathcal{O}_0}} ( \lambda \vert \nabla \varphi \vert)^3 \vert v \vert^2 \, dx .
\end{multline}
Next, since $v =  e^{\lambda\varphi} u$, we have $ e^{2\lambda\varphi}  \vert \nabla u \vert^2 \leqslant 2 \big(  \vert \nabla v \vert^2 + \lambda^2 \vert \nabla \varphi \vert^2 \vert v \vert^2  \big)$.
So it follows from $(\ref{eq8})$ and point $(a)$ in Lemma \ref{phi} that
\begin{multline*}
\dfrac{1}{C_0} \Big( \int_{{\mathcal{O}_0}} \vert \overset{\sim}{P}_\lambda v \vert^2 \, dx  + 2 \lambda \int_{\Gamma_0 \times \R} \partial_\nu \varphi \vert \partial_\nu v \vert^2 \, d\sigma_x \Big) \\
\geqslant \dfrac{1}{2}  \int_{\mathcal{O}_0} \big( \lambda \vert \nabla \varphi \vert \big)^{-1}\vert\Delta v \vert^2 \, dx +\dfrac{C_2 \lambda}{2} \int_{\mathcal{O}_0} e^{2\lambda\varphi} \vert  \nabla \varphi \vert \vert \nabla u \vert^2 \, dx + \lambda^3 C_2 (\beta - \alpha^{-1}) \int_{{\mathcal{O}_0}} ( \vert \nabla \varphi \vert)^3 \vert v \vert^2 \, dx .
\end{multline*}
Thus, we get upon possibly substituting $C_2( \alpha^{-1} +1)$ for $\beta_0$ that 
\begin{multline*}
\dfrac{1}{C_0} \Big( \int_{{\mathcal{O}_0}} \vert \overset{\sim}{P}_\lambda v \vert^2 \, dx  + 2 \lambda \int_{\Gamma_0 \times \R}  \partial_\nu \varphi  \vert \partial_\nu v \vert^2 \, d\sigma_x \Big) \\
\geqslant \dfrac{1}{2}  \int_{\mathcal{O}_0} \big( \lambda \vert \nabla \varphi \vert \big)^{-1}\vert\Delta v\vert^2 \, dx +\dfrac{C_2 \lambda}{2} \int_{\mathcal{O}_0} e^{2\lambda\varphi} \vert  \nabla \varphi \vert \vert \nabla u \vert^2 \, dx + \lambda^3 \int_{{\mathcal{O}_0}} ( \vert \nabla \varphi \vert)^3 \vert v \vert^2 \, dx .
\end{multline*}
Further, due to $(\ref{eq1})$ and $(C2)$, the estimate $\vert \partial_\nu \varphi \vert \leqslant C_4$ holds true in ${\mathcal{O}_0}$ with $C_4$ depending only on $\beta$, $\psi$ and $\alpha_0$.
Therefore, there exists a positive constant $C_5$ that depend only on $\beta$, $\psi$ and $\alpha_0$ such that 
$$
\int_{{\mathcal{O}_0}} \vert \overset{\sim}{P}_\lambda v \vert^2 \, dx  +  \lambda \int_{\Gamma_0 \times \R}  \vert \partial_\nu v \vert^2 \, d\sigma_x 
\geqslant C_5 \Big( \int_{\mathcal{O}_0}  \lambda^{-1}\vert \Delta v \vert^2 \, dx +\lambda \int_{\mathcal{O}_0} e^{2\lambda\varphi}  \vert \nabla u \vert^2 \, dx + \lambda^3 \int_{{\mathcal{O}_0}} \vert v \vert^2 \, dx \Big) .
$$
Finally, the Carleman's estimate follows immediately from this upon remembering that $v =  e^{\lambda\varphi} u$, $v = 0$ on $\partial {\mathcal{O}_0}$, $\vert P_\lambda v \vert^2 = e^{2\lambda\varphi}\vert \Delta u\vert^2$ and $\displaystyle\int_{\mathcal{O}_0} \vert R v \vert^2 \, dx \leqslant C \lambda^2 \int_{\mathcal{O}_0} \vert  v \vert^2 \, dx$.
\end{proof}
\subsection{Weak unique continuation property}
Armed with the Carleman's estimate that has just been proven, we can return now to the proof of the weak unique continuation property which is a standard and important tool for the proof of the stability estimate (see \cite{3,29}).  
\begin{proof}[Proof of Lemma \ref{UCP}]
Let $\psi_0$ be the function defined in Lemma A.1.
Since $\psi_0(x')>0$ for all $x' \in \mathcal{W}_0$, there exists a constant $\kappa >0$ such that
\begin{equation} \label{eq13}
\psi_0(x') \geqslant 2 \kappa; \quad x' \in \mathcal{W}_2 \backslash \mathcal{W}_3  .
\end{equation} 
Moreover, as $\psi_0(x')=0$, $x'\in \Gamma^\sharp $, there exist $\mathcal{W}^\sharp $ a small neighborhood of $\Gamma^\sharp $ such that 
\begin{equation} \label{eq14}
\psi_0(x') \leqslant \kappa; \quad  x' \in \mathcal{W}^\sharp ,  \quad \mathcal{W}^\sharp \cap \overline{\mathcal{W}}_1 = \varnothing.
\end{equation}
Let $\overset{\sim}{\mathcal{W}}^\sharp \subset \mathcal{W}^\sharp$ be an arbitrary neighborhood of $\Gamma^\sharp $. To apply $(\ref{eq2})$, it is necessary to introduce a function $\Theta$ satisfying $0\leqslant \Theta \leqslant 1$, $\Theta \in \mathcal{C}^\infty (\R^2)$ and 
\begin{equation}\label{eq15}
\Theta(x')=\left\lbrace
\begin{array}{ll}
1 & \mbox{in $\mathcal{W}_0 \backslash \mathcal{W}^\sharp$,}\\
0 & \mbox{in $\overset{\sim}{\mathcal{W}}^\sharp$.}
\end{array}
\right.
\end{equation}
Let $w$ be a solution to $(\ref{eq11})$. Setting $$w_1(x',x_3)= \Theta(x')w(x',x_3), \quad x' \in \omega, \, x_3 \in \R,$$ 
we get
$$\left\lbrace
\begin{array}{l}
\text{$(-\Delta + q_1)w_1(x) = \Theta(x)F(x)+Q_1(x,D)w \quad \quad \quad \quad $ in $ \mathcal{O}_0$, } \\
\text{$w_1=0 \quad \quad \quad  \quad \quad \,\quad \quad \quad \quad \quad \quad \quad \quad  \quad \quad \quad \quad \quad  \quad \,  $ on $ \partial \mathcal{O}_0$, } 
\end{array}\right.$$
where $Q_1(x,D)$ is a first order operator supported in $ \overline{\mathcal{W}^\sharp} \backslash \overset{\sim}{\mathcal{W}}^\sharp$ and given by
$$ Q_1(x,D)w  = [\Delta',\Theta]w.$$ 
By applying Carleman estimate $(\ref{eq2})$ to $w_1$, we obtain
\begin{multline}\label{eq16}
\lambda \int_{\mathcal{O}_0} e^{2\lambda\varphi} \big( \lambda^2 \vert w_1 \vert^2 + \vert \nabla w_1 \vert^2 \big) \, dx \\
\leqslant C \Big( \int_{\mathcal{O}_0} e^{2\lambda\varphi} \Big( \big\vert Q_1(x,D)w \big\vert^2 + \big\vert F(x) \big\vert^2 \Big) \, dx + \lambda \int_{\Gamma_0 \times \R} \big\vert \partial_\nu w_1 \big\vert^2 e^{2\lambda\varphi} \, d\sigma_{x} \Big).
\end{multline}
Let $ \mathcal{O}^\sharp = \mathcal{W}^\sharp \times \R$ and $\overset{\sim}{\mathcal{O}}^\sharp = \overset{\sim}{\mathcal{W}}^\sharp \times \R $. Using the fact that $Q_1(x,D)$ is a first order operator supported in $ \overline{\mathcal{O}^\sharp} \backslash \overset{\sim}{\mathcal{O}}^\sharp$ and by $(\ref{eq14})$, we get
\begin{align*}
\int_{\mathcal{O}_0} e^{2\lambda\varphi} \big\vert Q_1(x,D)w \big\vert^2 \, dx & \leqslant \int_{\mathcal{O}_0} e^{2\lambda e^{\beta\psi(x)}} \big\vert Q_1(x,D)w \big\vert^2 \, dx \\
& \leqslant e^{2\lambda e^{\beta\kappa}}  \int_{\mathcal{O}^\sharp \backslash \overset{\sim}{\mathcal{O}}^\sharp} \big\vert Q_1(x,D)w \big\vert^2 \, dx \\
& \leqslant C e^{2\lambda e^{\beta\kappa}}  \int_{\mathcal{O}^\sharp \backslash \overset{\sim}{\mathcal{O}}^\sharp} \big( \vert w \vert^2 + \vert \nabla w \vert^2 \big)  \, dx.
\end{align*}
On the other hand, by using the definition of $\Theta$ given by $(\ref{eq15})$, the estimate $(\ref{eq16})$ becomes:
\begin{multline*}
\lambda \int_{\mathcal{O}_0 \backslash \overset{\sim}{\mathcal{O}}^\sharp} e^{2\lambda\varphi}  \big( \lambda^2 \vert w \vert^2 + \vert \nabla w \vert^2 \big)  \, dx \leqslant  C \Big( e^{2\lambda e^{\beta\kappa}} \int_{\mathcal{O}^\sharp \backslash \overset{\sim}{\mathcal{O}}^\sharp} \big( \vert w \vert^2 + \vert \nabla w \vert^2 \big)  \, dx \\
+ \int_{\mathcal{O}_0} e^{2\lambda\varphi} \big\vert F(x) \big\vert^2 \, dx + \lambda \int_{\Gamma_0 \times \R} \big\vert \partial_{\nu} w \big\vert^2 e^{2\lambda\varphi} \, d\sigma_{x} \Big).
\end{multline*}
Moreover, by the fact that $\mathcal{O}_2 \backslash \mathcal{O}_3 \subset \mathcal{O}_0 \backslash \overset{\sim}{\mathcal{O}}^\sharp$ and by $(\ref{eq13})$, we easily obtain that:
\begin{multline*}
e^{2\lambda e^{2\beta\kappa}} \lambda \int_{\mathcal{O}_2 \backslash \mathcal{O}_3 }  \big( \lambda^2 \vert w \vert^2 + \vert \nabla w \vert^2 \big)  \, dx \leqslant  C \Big( e^{2\lambda e^{\beta\kappa}} \int_{\mathcal{O}^\sharp \backslash \overset{\sim}{\mathcal{O}}^\sharp} \big( \vert w \vert^2 + \vert \nabla w \vert^2 \big)  \, dx \\
+ \int_{\mathcal{O}_0} e^{2\lambda\varphi} \big\vert F(x) \big\vert^2 \, dx + \lambda \int_{\Gamma_0 \times \R} \big\vert \partial_{\nu} w \big\vert^2 e^{2\lambda\varphi} \, d\sigma_{x} \Big).
\end{multline*}
Thus, we have:
\begin{multline*}
\lambda \int_{\mathcal{O}_2 \backslash \mathcal{O}_3 }  \big( \lambda^2 \vert w \vert^2 + \vert \nabla w \vert^2 \big)  \, dx \leqslant  C \Big( e^{-2\lambda ( e^{2\beta\kappa} - e^{\beta\kappa})} \int_{\mathcal{O}^\sharp \backslash \overset{\sim}{\mathcal{O}}^\sharp} \big( \vert w \vert^2 + \vert \nabla w \vert^2 \big)  \, dx \\
+e^{2\lambda( e^{2\beta\Vert\psi_0\Vert_\infty} - e^{2\beta\kappa})} \Big( \int_{\mathcal{O}_0}  \big\vert F(x) \big\vert^2 \, dx + \lambda \int_{\Gamma_0 \times \R} \big\vert \partial_{\nu} w \big\vert^2  \, d\sigma_{x} \Big) \Big).
\end{multline*}
Let $ \alpha_1= ( e^{2\beta\kappa} - e^{\beta\kappa}) \, >0$ and $\alpha_2=( e^{2\beta\Vert\psi_0\Vert_\infty} - e^{2\beta\kappa}) \, >0$. We conclude that for any $\lambda>\lambda^*$, we have:
\begin{multline*}
\lambda \int_{\mathcal{O}_2 \backslash \mathcal{O}_3 }  \big( \lambda^2 \vert w \vert^2 + \vert \nabla w \vert^2 \big)  \, dx \leqslant  C \Big( e^{-2\lambda \alpha_1} \int_{\mathcal{O}^\sharp \backslash \overset{\sim}{\mathcal{O}}^\sharp} \big( \vert w \vert^2 + \vert \nabla w \vert^2 \big)  \, dx \\
+e^{2\lambda \alpha_2} \Big( \int_{\mathcal{O}_0}  \big\vert F(x) \big\vert^2 \, dx + \int_{\Gamma_0 \times \R} \big\vert \partial_{\nu} w \big\vert^2 \, d\sigma_{x} \Big) \Big).
\end{multline*}
Then, we have
$$ \Vert w \Vert_{H^1(\mathcal{O}_2 \backslash \mathcal{O}_3)}^2 \leqslant C \Big( e^{-2\lambda \alpha_1} \Vert w \Vert_{H^1(\Omega)}^2 + e^{2\lambda \alpha_2} \Big( \Vert F \Vert_{L^2(\mathcal{O}_0)}^2  + \big\Vert \partial_{\nu} w \big\Vert_{L^2(\Gamma_0 \times \R)}^2 \Big) \Big) $$
which completes the demonstration.

\end{proof}
\section*{Acknowledgments}
The author would like to thank Mr Mourad BELLASSOUED and Mr Yavar KIAN for their valuable and constructive suggestions during the planning and development of this research work and for the huge time spent to read carefully this manuscript. \\
The author would like also to thank the editor and the referees for the careful reading of the paper and for their insightful comments on which substantially helped improving the quality of the paper.

\end{document}